\documentclass[onefignum,onetabnum]{siamart171218}

\usepackage{amsfonts}
\usepackage{amsmath}
\usepackage{amssymb}
\usepackage{mathtools}
\usepackage{graphicx}
\usepackage{epstopdf}
\usepackage{algorithmic}
\usepackage{bbm}
\usepackage{lscape}
\usepackage{subfigure}
\ifpdf
  \DeclareGraphicsExtensions{.eps,.pdf,.png,.jpg}
\else
  \DeclareGraphicsExtensions{.eps}
\fi

\DeclareSymbolFont{lettersA}{U}{txmia}{m}{it}
\DeclareMathSymbol{\R}{\mathord}{lettersA}{"92}
\DeclareMathSymbol{\C}{\mathord}{lettersA}{"83}

\newcommand{\fast}{\{F\}}
\newcommand{\slow}{\{S\}}
\newcommand{\eslow}{\{E\}}
\newcommand{\islow}{\{I\}}
\newcommand{\esloweslow}{\{E,E\}}
\newcommand{\eslowfast}{\{E,F\}}
\newcommand{\islowislow}{\{I,I\}}
\newcommand{\islowfast}{\{I,F\}}
\newcommand{\eslowislow}{\{E,I\}}
\newcommand{\isloweslow}{\{I,E\}}
\newcommand{\fastfast}{\{F,F\}}

\newcommand{\slowfast}{\{S,F\}}
\newcommand{\fasteslow}{\{F,E\}}
\newcommand{\fastislow}{\{F,I\}}
\newcommand{\ffast}{f^{\fast}}
\newcommand{\ffasti}{f^{\{F,i\}}}
\newcommand{\fasti}{\{F,i\}}
\newcommand{\fslow}{f^{\slow}}
\newcommand{\feslow}{f^{\eslow}}
\newcommand{\fislow}{f^{\islow}}

\newcommand{\sslowt}{\tilde{s}^{\slow}}
\newcommand{\sslow}{s^{\slow}}
\newcommand{\sfast}{s^{\fast}}

\newcommand{\sumjtoi}{\sum\limits_{j=1}^i}
\newcommand{\Abf}{\mathbf{A}}
\newcommand{\bbf}{\mathbf{b}}
\newcommand{\cbf}{\mathbf{c}}

\allowdisplaybreaks

\newsiamremark{remark}{Remark}
\newsiamremark{hypothesis}{Hypothesis}
\crefname{hypothesis}{Hypothesis}{Hypotheses}
\newsiamthm{claim}{Claim}

\headers{Implicit-Explicit Multirate Infinitesimal GARK Methods}{R.~Chinomona and D.~R.~Reynolds}

\title{Implicit-Explicit Multirate Infinitesimal GARK Methods \thanks{Submitted to the editors DATE.
\funding{Support for this work was provided by the Scientific Discovery through Advanced Computing (SciDAC) project ``Frameworks, Algorithms and Scalable Technologies for Mathematics (FASTMath),'' funded by the U.S. Department of Energy Office of Advanced Scientific Computing Research and National Nuclear Security Administration, under Lawrence Livermore National Laboratory subcontract B626484 and DOE award DE-SC0021354.}}}

\author{
Rujeko Chinomona\thanks{Department of Mathematics, Southern Methodist University, Dallas, TX (\email{rchinomona@smu.edu}, \email{reynolds@smu.edu}).}
\and
Daniel R.~Reynolds\footnotemark[2]}

\usepackage{amsopn}

\makeatletter
\newcommand*{\addFileDependency}[1]{
  \typeout{(#1)}
  \@addtofilelist{#1}
  \IfFileExists{#1}{}{\typeout{No file #1.}}
}
\makeatother

\ifpdf
\hypersetup{
  pdftitle={Implicit-Explicit Multirate Infinitesimal GARK Methods},
  pdfauthor={R. Chinomona and D.R. Reynolds}
}
\fi

\begin{document}

\maketitle
\begin{abstract}
This work focuses on the development of a new class of high-order accurate methods for multirate time integration of systems of ordinary differential equations.  Unlike other recent work in this area, the proposed methods support mixed implicit-explicit (IMEX) treatment of the slow time scale.  In addition to allowing this slow time scale flexibility, the proposed methods utilize a so-called `infinitesimal' formulation for the fast time scale through definition of a sequence of modified `fast' initial-value problems, that may be solved using any viable algorithm.  We name the proposed class as implicit-explicit multirate infinitesimal generalized-structure additive Runge--Kutta (IMEX-MRI-GARK) methods.  In addition to defining these methods, we prove that they may be viewed as specific instances of GARK methods, and derive a set of order conditions on the IMEX-MRI-GARK coefficients to guarantee both third and fourth order accuracy for the overall multirate method.  Additionally, we provide three specific IMEX-MRI-GARK methods, two of order three and one of order four.  We conclude with numerical simulations on two multirate test problems, demonstrating the methods' predicted convergence rates and comparing their efficiency against both legacy IMEX multirate schemes and recent third and fourth order implicit MRI-GARK methods.
\end{abstract}




\section{Introduction}

In recent years, there has been a renewed interest in time integration methods, most notably those that allow both high accuracy and increased flexibility with regard to how various components of the problem are treated.  These methods range from those that apply a uniform time step size for all components of a problem but vary the algorithms used on individual terms, to `multirate' methods that evolve separate solution components using different step sizes.

Methods in the former category have been introduced primarily to handle problems that couple stiff and nonstiff processes. Here, instead of applying a fully implicit or fully explicit treatment, that would be ideally suited to only the stiff or nonstiff components of the problem, respectively, these approaches allow more robust implicit solvers to be applied to the stiff components, leaving the remaining nonstiff (and frequently nonlinear) components to be treated explicitly.  Various techniques within this category include mixed implicit-explicit (IMEX) additive Runge--Kutta methods \cite{ascherImplicitexplicitRungeKuttaMethods1997,cooperAdditiveMethodsNumerical1980,cooperAdditiveRungeKuttaMethods1983,kennedyAdditiveRungeKutta2003,kennedyHigherorderAdditiveRunge2019,sanduGeneralizedStructureApproachAdditive2015}, exponential Runge--Kutta (ExpRK) and exponential Rosenbrock (ExpRB) methods \cite{hochbruckExplicitExponentialRunge2005,luanExplicitExponentialRunge2014,luanExponentialRosenbrockMethods2014,rainwaterNewClassSplit2014,tokmanNewAdaptiveExponential2012a,tokmanNewClassExponential2011} and general linear methods (GLM) \cite{cardoneConstructionHighlyStable2015,cardoneExtrapolationbasedImplicitexplicitGeneral2014,robertsParallelImplicitexplicitGeneral2020,zhangHighOrderImplicitexplicit2016,zhangPartitionedImplicitExplicit2014}.

Multirate methods, on the other hand, evolve separate solution components or dynamical processes using entirely different time step sizes.  These frequently arise due to `multiphysics' problems wherein separate physical processes evolve on disparate time scales.  Either due to stability or accuracy considerations the `fast' processes must be evolved with small step sizes, but due to their computational cost the `slow' processes are evolved using sometimes much larger time steps.  While simplistic low-order `subcycling' approaches have been employed in computational simulations for decades, research into higher-order approaches has seen dramatic recent advances \cite{Bauer2019Extended,constantinescuExtrapolatedMultirateMethods2013,gearMultirateLinearMultistep1984,guntherMultirateGeneralizedAdditive2016,luanNewClassHighOrder2020,robertsCoupledMultirateInfinitesimal2020,robertsImplicitMultirateGARK2019,Sandu2019,schlegelMultirateRungeKutta2009,schlegelNumericalSolutionMultiscale2012,sextonRelaxedMultirateInfinitesimal2019,wenschMultirateInfinitesimalStep2009}.

In this paper we introduce a hybrid of two of the above techniques: IMEX Runge--Kutta and multirate methods.  While the large majority of recent research on multirate methods has focused on the two-way, additive initial-value problem (IVP) combining a fast $\fast$ and a slow $\slow$ process,
\begin{equation}
    y' = f(t,y) = \ffast(t,y) + \fslow(t,y),\quad t\in[t_0,t_f], \quad y(t_0)=y_0,
    \label{eq:usualform}
\end{equation}
we focus on problems that further break down the slow portion into stiff $\islow$ and nonstiff $\eslow$ components.  Thus we consider the three-way additive IVP:
\begin{equation}
    y' = \fislow(t,y) + \feslow(t,y) + \ffast(t,y),\quad t\in[t_0,t_f], \quad y(t_0)=y_0.
    \label{eq:probode}
\end{equation}
Of the various approaches for multirate integration, we focus on those that are agnostic as to the precise methods applied to the fast dynamics.  These are based on `infinitesimal' formulations, including the seminal work on multirate infinitesimal step (MIS) methods \cite{schlegelMultirateRungeKutta2009,wenschMultirateInfinitesimalStep2009} and their more recent extensions to higher temporal order \cite{Bauer2019Extended,luanNewClassHighOrder2020,robertsCoupledMultirateInfinitesimal2020,Sandu2019,sextonRelaxedMultirateInfinitesimal2019}.
In such formulations, the fast dynamics are assumed to be solved `exactly', typically through evolution of a sequence of modified fast IVPs,
\[
   v'(\theta) = \ffast(\theta,v) + g(\theta),\quad \theta\in[\theta_0,\theta_f], \quad v(\theta_0)=v_0,
\]
where the forcing function $g(\theta)$ is determined by the multirate method to incorporate information from $\fslow$.  In practice, these fast IVPs are solved using another numerical method with smaller step size, which in turn could employ further decompositions via an IMEX, ExpRK, ExpRB, GLM, or multirate approach.

To our knowledge, there exist only two multirate schemes that simultaneously allow IMEX treatment of the slow dynamics and infinitesimal treatment of the fast dynamics, both of which have low accuracy and have been shown to demonstrate poor stability \cite{estepPosterioriPrioriAnalysis2008,roppStabilityOperatorSplitting2005}.  The first of these is the standard first order ``Lie--Trotter'' splitting that performs the time step $y_{n} \to y_{n+1}$ (here $y_n \approx y(t_n)$ and $t_{n+1}-t_{n}=H$) \cite{mclachlanSplittingMethods2002} via the algorithm:
\begin{align}
\label{eq:first_order_split}
    &y_{n+1}^{(1)} = y_{n} + H \feslow(t_{n},y_{n}), \\
    &y_{n+1}^{(2)} = y_{n+1}^{(1)} + H \fislow(t_{n+1},y_{n+1}^{(2)}), \notag \\
    &\text{Solve } \begin{cases}
        \;v(0) = y_{n+1}^{(2)},\\
        \;v'(\theta) = \ffast(t_{n}+\theta,v), \;\text{for } \theta\in[0,H],
    \end{cases} \notag\\
    &y_{n+1} = v(H) \notag.
\end{align}
The second is a variant of the second order ``Strang'' (or ``Strang--Marchuk'') splitting formulation \cite{Marchuk1968,strangConstructionComparisonDifference1968a},
\begin{align}
    \label{eq:Strang_split}
    &y_{n+1}^{(1)} = y_{n} + \tfrac{H}{4} \feslow\left(t_{n},y_{n}\right)\\
    &\qquad + \tfrac{H}{4} \feslow\left(t_{n}+\tfrac{H}{2},y_{n}+\tfrac{H}{2}\feslow\left(t_{n},y_{n}\right)\right),\notag \\
    &y_{n+1}^{(2)} = y_{n+1}^{(1)} + \tfrac{H}{4} \fislow\left(t_{n},y_{n+1}^{(1)}\right) + \tfrac{H}{4} \fislow\left(t_{n}+\tfrac{H}{2},y_{n+1}^{(2)}\right), \notag \\
    &\text{Solve } \begin{cases}
        \;v(0) = y_{n+1}^{(2)},\\
        \;v'(\theta) = \ffast\left(t_n+\theta,v\right), \;\text{for } \theta\in[0,H],
    \end{cases} \notag \\
    &y_{n+1}^{(3)} = v(H), \notag\\
    &y_{n+1}^{(4)} = y_{n+1}^{(3)} + \tfrac{H}{4} \fislow\left(t_{n}+\tfrac{H}{2},y_{n+1}^{(3)}\right) + \tfrac{H}{4} \fislow\left(t_{n+1},y_{n+1}^{(4)}\right), \notag \\
    &y_{n+1} = y_{n+1}^{(4)} + \tfrac{H}{4} \feslow\left(t_{n}+\tfrac{H}{2},y_{n+1}^{(4)}\right) \notag\\
    &\qquad + \tfrac{H}{4} \feslow\left(t_{n+1},y_{n+1}^{(4)}+\tfrac{H}{2}\feslow\left(t_{n}+\tfrac{H}{2},y_{n+1}^{(4)}\right)\right)\notag.
\end{align}
We note that here, the updates $y_{n} \to y_{n+1}^{(1)}$ and $y_{n+1}^{(4)} \to y_{n+1}$ correspond to using the explicit Heun method for a half time-step each, while the updates $y_{n+1}^{(1)} \to y_{n+1}^{(2)}$ and $y_{n+1}^{(3)} \to y_{n+1}^{(4)}$ correspond to using the implicit trapezoid rule for a half time-step each.
However to our knowledge, there do not exist multirate methods allowing IMEX treatment of the slow time scale that have order of accuracy three or higher.  The purpose of this paper is to address this need, through proposal of a new class of \emph{implicit-explicit multirate infinitesimal generalized-structure additive Runge--Kutta} (IMEX-MRI-GARK) methods for problems of the form \eqref{eq:probode}, including derivation of order conditions up to fourth order, and numerical tests to demonstrate the benefit of such methods over the legacy approaches \eqref{eq:first_order_split} and \eqref{eq:Strang_split}, as well as to provide comparisons against recent third and fourth order implicit MRI-GARK methods.


\section{Implicit-Explicit Multirate Infinitesimal GARK Methods}
\label{section2}
We build our proposed methods by extending the MRI-GARK class of two-component multirate methods \cite{Sandu2019}. An MRI-GARK method with $\sslow$ slow stages is uniquely defined by the abcissae $c^{\slow}\in\R^{s^{\slow}}$ and $(k_{\max}+1)$ lower-triangular matrices of coefficients $\Gamma^{\{k\}}\in\R^{s^{\slow}\times s^{\slow}}$.
One step of an MRI-GARK applied to \eqref{eq:usualform} from $t_n$ to $t_n+H$ is defined by the algorithm
\begin{subequations}
\label{mrigark}
\begin{align}
  &\text{Let}:\; Y_1^{\slow} := y_n \label{eq:mriga} \\
  &\text{For } i = 2,\ldots,\sslow:\\
  &\label{eq:mrigb} \begin{cases}
            \text{Let:} &v(0) := Y_{i-1}^{\slow} \quad\text{and}\quad T_{i-1} := t_n + c_{i-1}^{\slow} H,\\
            \text{Solve:} &v'(\theta) = \Delta c_i^{\slow} \ffast \left(T_{i-1} + \Delta c_i^{\slow} \theta,\, v(\theta)\right) + g(\theta), \text{ for } \theta \in [0,H], \\
            &\text{where } g(\theta) = \sumjtoi \gamma_{i,j}\!\left(\frac{\theta}{H}\right)\, \fslow\!\left(t_n + c_j^{\{S\}}H,Y^{\slow}_j \right) \\
            \text{Let:} &Y_{i}^{\slow} := v(H),
        \end{cases}\\
  &\text{Let}:\; y_{n+1} := Y_{\sslow}^{\slow}. \label{eq:mrigc}
\end{align}
\end{subequations}
Here the abcissae satisfy $0=c_1^{\slow} \le c_{2}^{\slow} \le \cdots \le c_{s^{\slow}}^{\slow} \le 1$, and the increments are given by $\Delta c_{i}^{\slow} := c_{i}^{\slow} - c_{i-1}^{\slow} \ge 0, i = 2,\ldots,\sslow$.  The modified fast IVPs \eqref{eq:mrigb} combine the fast component $\ffast$ and a forcing term $g(\theta)$, and serve to advance the solution between slow stages $Y_i^{\slow}$. The slow tendency terms $\gamma_{i,j}(\tau)$ are polynomials in time that dictate the couplings from the slow to the fast time scale, and are defined from the coefficients $\Gamma^{\{k\}}$ as
\begin{equation}
\label{eq:gammacoeffs1}
    \gamma_{i,j}(\tau) \coloneqq \sum_{k=0}^{k_{\max}} \gamma_{i,j}^{\{k\}} \tau^k.
\end{equation}
These coefficients are derived from order conditions for MRI-GARK methods, and essentially serve to interpolate information from the slow to the fast time scale.  For example, the MRI-GARK-ERK33a method from \cite{Sandu2019} is defined through the coefficients
\[
   c^{\slow} = \begin{bmatrix} 0\\\tfrac13\\\tfrac23\\1 \end{bmatrix}, \quad
   \Gamma^{\{0\}} = \begin{bmatrix} 0&0&0&0\\ \tfrac{1}{3}& 0&0&0\\ -\tfrac{1}{3}&\tfrac{2}{3}&0&0\\ 0& -\tfrac{2}{3} & 1& 0\end{bmatrix}, \quad
   \Gamma^{\{1\}} = \begin{bmatrix} 0&0&0&0\\ 0&0&0&0\\ 0&0&0&0\\ \tfrac{1}{2}&0&-\tfrac{1}{2}&0\end{bmatrix}.
\]
Here, the strictly lower triangular structure of the two $\Gamma$ matrices ensure that each MRI-GARK stage \eqref{eq:mrigb} is explicit, in that the forcing function $g(\theta)$ is uniquely defined from previous slow stages $Y^{\slow}_j$.  Thus one time step of MRI-GARK-ERK33a requires the solution of three modified IVPs \eqref{eq:mrigb}, and three evaluations of $f^{\slow}$.

We note that our presentation of MRI-GARK methods above differs slightly from \cite{Sandu2019}, in that we include the zero-valued coefficients of the first stage in our tables, with the effect that $\Delta c^{\slow}_{1} = 0$, and `implicitness' of an MRI-GARK method is indicated by nonzero $\gamma_{i,j}^{\{k\}}$ values on or above the diagonal, as is typically denoted for Runge--Kutta methods.

\begin{definition}[IMEX-MRI-GARK methods for additive systems]
\label{def:IMEX-MRI-method}
Let $c^{\slow} \in\R^{\sslow}$ with $c^{\slow}_1=0$ and $\Delta c_i^{\slow} = c^{\slow}_{i} - c^{\slow}_{i-1}\ge 0$ for $i=2,\ldots,\sslow$.  For $k=0,\ldots,k_{max}$, let $\Gamma^{\{k\}} \in \R^{\sslow\times\sslow}$ be lower triangular, and let $\Omega^{\{k\}} \in \R^{\sslow\times\sslow}$ be strictly lower triangular, with each having first row identically zero.  One step from $t_n$ to $t_{n+1} = t_n + H$ of an IMEX-MRI-GARK method for the problem \eqref{eq:probode} proceeds as
\begin{subequations}
\label{mriimexs}
\begin{align}
  &\text{Let}:\; Y_1^{\slow} := y_n \label{eq:a} \\
  &\text{For } i = 2,\ldots,\sslow:\\
  &\label{eq:b} \begin{cases}
            \text{Let:} &v(0) := Y_{i-1}^{\slow} \quad\text{and}\quad T_{i-1} := t_n + c_{i-1}^{\slow} H,\\
            \text{Solve:} &v'(\theta) = \Delta c_i^{\slow} \ffast \left(T_{i-1} + \Delta c_i^{\slow} \theta,\, v(\theta)\right)  + g(\theta), \text{ for } \theta \in [0,H], \\
            &\text{where } g(\theta) = \sumjtoi \gamma_{i,j}\!\left(\frac{\theta}{H}\right)\, \fislow_j + \sum\limits_{j=1}^{i-1} \omega_{i,j}\!\left(\frac{\theta}{H}\right)\, \feslow_j, \\
            \text{Let:} &Y_{i}^{\slow} := v(H),
        \end{cases}\\
  &\text{Let}: y_{n+1} := Y_{\sslow}^{\slow}. \label{eq:c}
\end{align}
\end{subequations}
Here $\fislow_j := \fislow\!\left(t_n + c_j^{\{S\}}H,Y_j^{\slow}\right)$ and $\feslow_j := \feslow\!\left(t_n + c_j^{\{S\}}H,Y_j^{\slow}\right)$. Similarly to MRI-GARK methods, the modified IVPs \eqref{eq:b} are defined using linear combinations of the slow implicit and slow explicit function values, where the polynomial coefficient functions are given by
\begin{equation}
    \label{eq:gammaomegacoeffs}
    \gamma_{i,j}(\tau) \coloneqq \sum_{k=0}^{k_{\max}} \gamma_{i,j}^{\{k\}} \tau^k
    \quad\text{and}\quad
    \omega_{i,j}(\tau) \coloneqq \sum_{k=0}^{k_{\max}} \omega_{i,j}^{\{k\}} \tau^k.
\end{equation}
\end{definition}

\subsection{Order Conditions}
In the same way MRI-GARK methods are derived by starting from an explicit or diagonally implicit Runge--Kutta method, IMEX-MRI-GARK methods may be derived by starting with an IMEX additive Runge-Kutta scheme (IMEX-ARK) of order $q$ and having $\sslowt$ stages, characterized by a pair of Butcher tables:
\begin{center}
\begin{tabular}{r|c}
  $c^{\eslow}$ & $A^{\eslow}$ \\
  \hline
               & $b^{\eslow T}$
\end{tabular}
\qquad
\begin{tabular}{r|c}
  $c^{\islow}$ & $A^{\islow}$ \\
  \hline
               & $b^{\islow T}$
\end{tabular}
\end{center}
For compatibility between the IMEX-ARK scheme and our eventual IMEX-MRI-GARK coefficients $c^{\slow}$, $\Gamma^{\{k\}}$ and $\Omega^{\{k\}}$, we only consider IMEX-ARK methods that are ``internally consistent,'' i.e., $c^{\eslow} = c^{\islow} := c^{\slow}$, and that have fully explicit first stage.  Additionally, to reduce complexity in our analyses we follow \cite{robertsCoupledMultirateInfinitesimal2020,Sandu2019} and write the base IMEX-ARK method in stiffly accurate form, i.e., the last row of $A^{\eslow}$ and $A^{\islow}$ equal $b^{\eslow T}$ and $b^{\islow T}$,  respectively.  We note that methods which do not satisfy this requirement in simplest form may easily be converted to the stiffly accurate form by padding $c$ and $A$ with $1$ and $b^T$, respectively:
\begin{center}
\begin{tabular}{r|c|c}
  $c^{\slow}$ & $A^{\eslow}$   & $A^{\islow}$ \\
  \hline
              & $b^{\eslow T}$ & $b^{\islow T}$
\end{tabular}
$\quad\rightarrow\quad$
\begin{tabular}{r|c l|c l}
  $c^{\slow}$ & $A^{\eslow}$ & $0^{\slow}$ & $A^{\islow}$ & $0^{\slow}$\\
  $1$ & $b^{\eslow T}$& $0$ & $b^{\islow T}$ & $0$ \\
  \hline
      & $b^{\eslow T}$& $0$ & $b^{\islow T}$ & $0$
\end{tabular}
\end{center}
where $0^{\slow}\in \mathbb{R}^{\sslowt}$. Thus for the remainder of this paper, we let $A^{\esloweslow}, A^{\islowislow} \in \mathbb{R}^{\sslow \times \sslow}$ be the stiffly-accurate versions of the IMEX-ARK Butcher tables $A^{\eslow}$ and $A^{\islow}$, respectively.  We note that this extension of the tables to include the row of $b$ coefficients does not affect the order conditions of the original IMEX-ARK table, and thus all order conditions satisfied by the original IMEX-ARK tables remain unchanged.

With these IMEX-ARK tables in place, we derive order conditions for the IMEX-MRI-GARK coefficients $c^{\slow}$, $\Gamma^{\{k\}}$ and $\Omega^{\{k\}}$ by first expressing IMEX-MRI-GARK methods in GARK form, following similar derivations applied to other infinitesimal methods \cite{Bauer2019Extended,robertsCoupledMultirateInfinitesimal2020,Sandu2019,sextonRelaxedMultirateInfinitesimal2019}; thus we first identify the GARK tables $\Abf^{\{\sigma,\nu\}}$, $\bbf^{\sigma}$ and $\cbf^{\sigma}$ for $\sigma,\nu\in\{I,E,F\}$.  To this end, we consider the inner modified fast IVP \eqref{eq:b} to be evolved using a single step of an arbitrary $\sfast$-stage Runge--Kutta method with Butcher table $(A^{\{F,F\}}, b^{\fast}, c^{\fast})$, having order of accuracy $q$ at least as accurate as the IMEX-MRI-GARK method. Thus the $k^{th}$ fast stage ($k = 1,\ldots, \sfast$) within the $i^{th}$ slow stage ($i=2,\ldots,\sslow$) is given by:
\begin{align}
 \label{eq:faststage}
    Y_k^{\fasti} = Y_{i-1}^{\slow} &+ H\Delta c_i^{\slow} \sum_{l=1}^{\sfast} a_{k,l}^{\{F,F\}}\, \ffasti_l\\
    &+ H \sum_{j=1}^{i}\left(\sum_{l=1}^{\sfast} a_{k,l}^{\{F,F\}} \gamma_{i,j}\!\left(c_l^{\fast}\right) \right) \fislow_j \notag \\
    &+ H \sum_{j=1}^{i-1}\left(\sum_{l=1}^{\sfast} a_{k,l}^{\{F,F\}} \omega_{i,j}\!\left(c_l^{\fast}\right) \right) \feslow_j, \notag
\end{align}
where $\ffasti_l := \ffast\!\left(T_{i-1} + c_l^{\fast} \Delta c_{i}^{\slow} H, Y_l^{\fasti}\right)$.  Similarly, the slow stages in this scenario become:
\begin{align}
\label{eq:yib4simplify}
    Y_{i}^{\slow} = Y_{i-1}^{\slow}
    &+ H \sum_{j=1}^{i}\left(\sum_{l=1}^{\sfast} b_{l}^{\fast} \gamma_{i,j}\!\left(c_l^{\fast}\right) \right) \fislow_j \\
    & + H \sum_{j=1}^{i-1}\left(\sum_{l=1}^{\sfast} b_{l}^{\fast} \omega_{i,j}\!\left(c_l^{\fast}\right) \right) \feslow_j \notag\\
    & + H\Delta c_{i}^{\slow} \sum_{l=1}^{\sfast} b_{l}^{\fast} \ffasti_l. \notag\\
    = y_n
     &+ H \sum_{\lambda = 1}^{i}  \sum_{j=1}^{\lambda}\left(\sum_{l=1}^{\sfast} \sum_{k=0}^{k_{\max}} \gamma_{\lambda,j}^{\{k\}} b_{l}^{\fast}c_l^{\fast \times k} \right) \fislow_j \notag \\
     &+ H \sum_{\lambda = 1}^{i}  \sum_{j=1}^{\lambda-1}\left(\sum_{l=1}^{\sfast} \sum_{k=0}^{k_{\max}} \omega_{\lambda,j}^{\{k\}} b_{l}^{\fast}c_l^{\fast \times k} \right) \feslow_j, \notag \\
    &+ H \sum_{\lambda = 1}^{i} \Delta c_{\lambda}^{\slow} \sum_{l=1}^{\sfast} b_{l}^{\fast} f_l^{\{F,\lambda\}},   \notag
\end{align}
due to \eqref{eq:gammacoeffs1}, and where we use the notation $c^{\times k}$ to indicate element-wise exponentiation.  Then using \eqref{eq:gammabarcoeffs} and our assumption that Runge--Kutta method for the fast partition satisfies $b^{\fast T}c^{\fast \times k} = 1/(k+1)$ for $k = 1,\ldots,q$, we simplify \eqref{eq:yib4simplify} to obtain:
\begin{align}
\label{eq:yslowstage}
  Y_{i}^{\slow} &= y_n + H \sum_{j = 1}^{i}  \sum_{\lambda = j}^{i}\overline{\gamma}_{\lambda,j} \fislow_j
     + H \sum_{j = 1}^{i-1}  \sum_{\lambda=j}^{i} \overline{\omega}_{\lambda,j} \feslow_j \\
     & \quad + H \sum_{\lambda = 1}^{i} \sum_{l=1}^{\sfast} \Delta c_{\lambda}^{\slow} b_{l}^{\fast} f_l^{\{F,\lambda\}}, \notag
\end{align}
where
\begin{equation}
    \label{eq:gammabarcoeffs}
    \overline{\gamma}_{i,j} \coloneqq \sum_{k=0}^{k_{\max}} \gamma_{i,j}^{\{k\}} \frac{1}{k+1}
    \quad\text{and}\quad
    \overline{\omega}_{i,j} \coloneqq \sum_{k=0}^{k_{\max}} \omega_{i,j}^{\{k\}} \frac{1}{k+1}.
\end{equation}
Recalling that the original IMEX-ARK method had an explicit first stage, \eqref{eq:yslowstage} is equivalent to the standard GARK formulation,
\begin{align}
\label{eq:yslowstage_GARK}
    Y_{i}^{\slow} &= y_n + H \sum_{j = 1}^{i}  a_{i,j}^{\islowislow} \fislow_j
    + H \sum_{j = 1}^{i-1} a_{i,j}^{\esloweslow}  \feslow_j
    + H \sum_{\lambda = 1}^{i} \sum_{j=1}^{\sfast} a_{i,j}^{\{S,F,\lambda\}} f_j^{\{F,\lambda\}},
\end{align}
for slow stages $i=1,\ldots,\sslow$, where we identify the slow implicit, slow explicit and slow-fast coupling coefficients as:
\begin{align}
\label{eq:GARK_coefficients_componentform}
    a_{i,j}^{\islowislow} := \sum_{\lambda=j}^i\overline{\gamma}_{\lambda,j}, \quad
    a_{i,j}^{\esloweslow} := \sum_{\lambda=j}^{i}\overline{\omega}_{\lambda,j}, \quad
    a_{i,j}^{\{S,F,\lambda\}} := \Delta c_{\lambda}^{\slow} b_{j}^{\fast}.
\end{align}
The first two of these may be represented as the GARK tables
\begin{align}
  \label{eq:AslowGARK}
  \Abf^{\islowislow} := E\overline{\Gamma} = A^{\islowislow} \quad\text{and}\quad
  \Abf^{\esloweslow} := E\overline{\Omega} = A^{\esloweslow},
\end{align}
where
\begin{equation*}
    E \in \mathbb{R}^{\sslow \times \sslow}, \quad
    E_{i,j} := \begin{cases} 1, & i\geq j, \\
                            0, & \text{otherwise}.\end{cases}
\end{equation*}
We note that
the conditions $E\overline{\Gamma} = A^{\islowislow}$  and $E\overline{\Omega} = A^{\esloweslow}$ in \eqref{eq:AslowGARK} ensure consistency between the IMEX-MRI-GARK method \eqref{mriimexs} and the underlying IMEX-ARK method in the non-multirate case where $\ffast \equiv 0$.

Furthermore, since the GARK formulation of standard IMEX-ARK methods satisfies $A^{\isloweslow} = A^{\esloweslow}$ and $A^{\eslowislow} = A^{\islowislow}$ (see \cite{sanduGeneralizedStructureApproachAdditive2015}), the GARK formulation of our IMEX-MRI-GARK method results in the slow explicit and slow implicit portions having \textit{shared} slow-fast coupling matrix $\Abf^{\eslowfast} = \Abf^{\islowfast} := \Abf^{\slowfast} \in \mathbb{R}^{\sslow\times s}$ with $s = \sfast\sslow$.  From \eqref{eq:GARK_coefficients_componentform}, we have the sub-matrices
\begin{equation}
  \Abf^{\{S,F,\lambda\}} := \Delta c^{\slow}_{\lambda}\, \mathbf{g}_{\lambda} \,b^{\fast T}, \quad\text{for}\quad \lambda = 1, \ldots, \sslow,
\end{equation}
where $\mathbf{g}_{\lambda} \in \mathbb{R}^{\sslow}$ with
\begin{align*}
    \Big(\mathbf{g}_{\lambda}\Big)_i := \begin{cases} 1, & i\geq \lambda,\\
        0, & \text{otherwise}. \end{cases}
\end{align*}
Combining these into an overall slow-fast coupling matrix, we have
\begin{align}
    \Abf^{\slowfast} &:= \begin{bmatrix} \Abf^{\{S,F,1\}}, & \cdots, & \Abf^{\{S,F,\sslow\}} \end{bmatrix} = \Delta C^{\slow} \otimes b^{\fast T},
\end{align}  where
\begin{equation*}
    \Delta C^{\slow} := \begin{bmatrix}
    \Delta c_1^{\slow} & 0^{\fast T} & \cdots & 0^{\fast T} \\
    \Delta c_1^{\slow} & \Delta c_2^{\slow} & \cdots & 0^{\fast T} \\
    \vdots & \vdots & \ddots & 0^{\fast T} \\
    \Delta c_1^{\slow} & \Delta c_2^{\slow} & \cdots & \Delta c_{\sslow}^{\slow}
    \end{bmatrix},
\end{equation*}
and $0^{\fast}$ is a column vector of all zeros in $\mathbb{R}^{\sfast}.$

For completeness, we note the corresponding GARK slow implicit and slow explicit coefficients \cite{Sandu2019},
\begin{align}
\label{eq:bislow}
  \bbf^{\islow} &:= \mathbbm{1}^{\slow T} \overline{\Gamma} = b^{\islow},\\
\label{eq:cislow}
  \cbf^{\islow} &:= E\overline{\Gamma}\mathbbm{1}^{\slow} = A^{\islowislow}\mathbbm{1}^{\slow} = c^{\slow},\\
\label{eq:beslow}
  \bbf^{\eslow} &:= \mathbbm{1}^{\slow T} \overline{\Omega} = b^{\eslow},\\
\label{eq:ceslow}
  \cbf^{\eslow} &:= E\overline{\Omega}\mathbbm{1}^{\slow} = A^{\esloweslow}\mathbbm{1}^{\slow} = c^{\slow},
\end{align}
where $\mathbbm{1}^{\slow} \in \mathbb{R}^{\sslow}$ is a column vector of all ones, and we have relied on our assumption of internal consistency in the underlying IMEX-ARK method.  From enforcing the row-sum conditions on $\Abf^{\slowfast}$, we have
\begin{align}
\label{eq:cslowfast}
  \cbf^{\slowfast} &:= \sum_{\lambda=1}^{\sslow} \Abf^{\{S,F,\lambda\}} \mathbbm{1}^{\fast} = \sum_{\lambda = 1}^{\sslow} \Delta c_{\lambda}\mathbf{g}_{\lambda}\\
\Rightarrow\qquad\qquad&\notag\\
 \cbf_i^{\slowfast} &= \sum_{\lambda = 1}^{\sslow} (c_{\lambda}^{\slow} - c_{\lambda -1}^{\slow})(\mathbf{g}_{\lambda})_i = \sum_{\lambda = 1}^{i} (c_{\lambda}^{\slow} - c_{\lambda -1}^{\slow}) = c_i^{\slow}, \notag
\end{align}
which ensures internal consistency between each partition of the GARK table (i.e., $\cbf^{\islowislow} = \cbf^{\esloweslow} = \cbf^{\slowfast} = c^{\slow}$).

To reveal the GARK coefficients for the fast method and fast-slow couplings, we insert \eqref{eq:yslowstage} into \eqref{eq:faststage} to write the $k^{th}$ fast stage $(k=1,\ldots,\sfast)$ within the $i^{th}$ slow stage $(i=2,\ldots,\sslow)$ as:
\begin{align}
    Y_k^{\fasti} &= y_n + H\sum_{\lambda = 1}^{i-1} \sum_{l=1}^{\sfast} \Delta c_{\lambda}^{\slow} b_l^{\fast} f_{l}^{\{F,\lambda\}}
    + H \Delta c_{i}^{\slow} \sum_{l =1}^{\sfast} a_{k,l}^{\fastfast} \ffasti_l \\
   \notag &\quad + H \sum_{j =1}^{i-1} a_{i-1,j}^{\islowislow} \fislow_j
    + H \sum_{j=1}^i \left(\sum_{l=1}^{\sfast} a_{k,l}^{\fastfast} \gamma_{i,j}\!\left(c_l^{\fast}\right) \fislow_j \right) \\
   \notag &\quad + H \sum_{j =1}^{i-2} a_{i-1,j}^{\esloweslow} \feslow_j
    + H \sum_{j=1}^{i-1} \left(\sum_{l=1}^{\sfast} a_{k,l}^{\fastfast} \omega_{i,j}\!\left(c_l^{\fast}\right) \feslow_j \right).
\end{align}
The fast method coefficients are therefore:
\begin{align}
  \label{eq:Afastfast}
     \Abf^{\fastfast} &:= \begin{bmatrix}
     \Delta c_{1}^{\slow} A^{\fastfast} & 0_{\sfast \times \sfast} & \cdots & 0_{\sfast \times \sfast}\\
     \Delta c_{1}^{\slow} \mathbbm{1}^{\fast} b^{\fast T} & \Delta c_{2}^{\slow} A^{\fastfast} & \cdots & \vdots \\
     \Delta c_{1}^{\slow} \mathbbm{1}^{\fast} b^{\fast T}  & \Delta c_{2}^{\slow} \mathbbm{1}^{\fast} b^{\fast T} & \cdots & \vdots \\
      \vdots & \vdots & \ddots & \vdots \\
      \Delta c_{1}^{\slow} \mathbbm{1}^{\fast} b^{\fast T} & \Delta c_{2}^{\slow} \mathbbm{1}^{\fast} b^{\fast T} & \cdots & \Delta c_{\sslow}^{\slow} A^{\fastfast}
     \end{bmatrix}\\
    &= \text{diag}\big(\Delta c^{\slow} \big) \otimes A^{\fastfast} + L \Delta C^{\slow}\otimes \mathbbm{1}^{\fast} b^{\fast T} \in \mathbb{R}^{s \times s}, \notag
\end{align}
where $\text{diag}\big(\Delta c^{\slow} \big)$ is the diagonal matrix obtained by taking $\Delta c^{\slow}$ as its diagonal entries, and where $L\in\R^{\sslow\times\sslow}$ has entries $L_{i,j} := \delta_{i,j+1}$;
similarly,
\begin{align}
   \label{eq:cfast}
   \cbf^{\fast} &:= \begin{bmatrix}
   \Delta c_1^{\slow} c^{\fast}\\ c_1^{\slow}\mathbbm{1}^{\fast} + \Delta c_{2}^{\slow} c^{\fast}\\ \vdots \\
   c_{\sslow-1}^{\slow}\mathbbm{1}^{\fast} + \Delta c_{\sslow}^{\slow} c^{\fast}     \end{bmatrix} = Lc^{\slow} \otimes \mathbbm{1}^{\fast} + \Delta c^{\slow} \otimes c^{\fast} \in \mathbb{R}^{s}\\
   \notag \text{and}\quad&\\
   \label{eq:bfast}
   \bbf^{\fast} &:= \begin{bmatrix} \Delta c_1^{\slow} b^{\fast}\\
   \vdots \\ \Delta c_{\sslow}^{\slow} b^{\fast} \end{bmatrix}  = \Delta c^{\slow} \otimes b^{\fast} \in \mathbb{R}^{s}.
\end{align}
%
%
Finally, the fast implicit and fast explicit coupling coefficients are
 \begin{align}
  \Abf^{\fastislow} &:= \begin{bmatrix}
  0_{\sfast \times \sslow}\\
     \mathbbm{1}^{\fast} \big(e_1^TA^{\islowislow}\big) + \sum\limits_{k=0 }^{k_{\max}} \big(A^{\fastfast} c^{\fast \times k}\big)\big(e_2^T {\Gamma}^{\{k\}}\big)
     \\
     \vdots
     \\
        \mathbbm{1}^{\fast} \big(e_{\sslow-1}^TA^{\islowislow}\big) + \sum\limits_{k=0}^{k_{\max}} \big(A^{\fastfast} c^{\fast \times k}\big)\big(e_{\sslow}^T {\Gamma}^{\{k\}}\big)
     \end{bmatrix}\\
     &= LA^{\islowislow} \otimes \mathbbm{1}^{\fast} + \sum_{k=0}^{k_{\max}} \Gamma^{\{k\}} \otimes \big(A^{\fastfast} c^{\fast \times k}\big) \in \mathbb{R}^{s \times \sslow}  \notag\\
   \notag \text{and}\quad&\\
   \Abf^{\fasteslow} &:=    \begin{bmatrix}
  0_{\sfast \times \sslow}\\
     \mathbbm{1}^{\fast} \big(e_1^TA^{\esloweslow}\big) + \sum\limits_{k=0}^{k_{\max}} \big(A^{\fastfast} c^{\fast \times k}\big)\big(e_2^T {\Omega}^{\{k\}}\big)
     \\
     \vdots
     \\
        \mathbbm{1}^{\fast} \big(e_{\sslow-1}^TA^{\esloweslow}\big) + \sum\limits_{k=0}^{k_{\max}} \big(A^{\fastfast} c^{\fast \times k}\big)\big(e_{\sslow}^T {\Omega}^{\{k\}}\big)
     \end{bmatrix}\\
     &= LA^{\esloweslow} \otimes \mathbbm{1}^{\fast} + \sum_{k=0}^{k_{\max}} \Omega^{\{k\}} \otimes \big(A^{\fastfast} c^{\fast \times k}\big) \in \mathbb{R}^{s \times \sslow}, \notag
\end{align}
where we have leveraged the fact that $\Gamma^{\{k\}}$ and $\Omega^{\{k\}}$ have zero first row.  These give rise to
\begin{align}
     \cbf^{\fastislow} &:= Lc^{\slow} \otimes \mathbbm{1}^{\fast} + \sum_{k=0}^{k_{\max}} \Gamma^{\{k\}} \mathbbm{1}^{\slow} \otimes (A^{\fastfast} c^{\fast \times k}),\quad\text{and}\\
     \cbf^{\fasteslow} &:= Lc^{\slow} \otimes \mathbbm{1}^{\fast} + \sum_{k=0}^{k_{\max}} \Omega^{\{k\}} \mathbbm{1}^{\slow} \otimes (A^{\fastfast} c^{\fast \times k}).
 \end{align}

 \begin{theorem}[Internal consistency conditions]
 \label{thm:internal_consistency}
 IMEX-MRI-GARK methods fulfill the ``internal consistency" conditions:
 \begin{align}
     \label{eq:cslowic} &\cbf^{\islowfast} = \cbf^{\eslowfast} = \cbf^{\slowfast} = \cbf^{\slow} \equiv c^{\slow}, \quad\text{and}\\
  \label{eq:cfastic} &  \cbf^{\fastislow} = \cbf^{\fasteslow} = \cbf^{\fast},
 \end{align}
 for any fast method if and only if the following conditions hold:
 \begin{equation}
 \label{eq:internalcons}
     \Gamma^{\{0\}} \mathbbm{1}^{\slow}  = \Omega^{\{0\}} \mathbbm{1}^{\slow} = \Delta c^{\slow} \quad \text{and} \quad \Gamma^{\{k\}} \mathbbm{1}^{\slow} = \Omega^{\{k\}} \mathbbm{1}^{\slow} = 0 \quad \forall k \geq 1.
 \end{equation}
 \end{theorem}
 \begin{proof}
 From the definition of $\cbf^{\slowfast}$ in equation \eqref{eq:cslowfast}, we have already shown that \eqref{eq:cslowic} is satisfied. Now
 \begin{align*}
   & \cbf^{\fastislow} = \cbf^{\fast} \Leftrightarrow \\
  & Lc^{\slow} \otimes \mathbbm{1}^{\fast} + \sum_{k=0}^{k_{\max}} \Gamma^{\{k\}} \mathbbm{1}^{\slow} \otimes (A^{\fastfast} c^{\fast \times k}) = Lc^{\slow} \otimes \mathbbm{1}^{\fast} + \Delta c^{\slow} \otimes c^{\fast},
 \end{align*}
 and similarly
  \begin{align*}
   & \cbf^{\fasteslow} = \cbf^{\fast} \Leftrightarrow \\
  & Lc^{\slow} \otimes \mathbbm{1}^{\fast} + \sum_{k=0}^{k_{\max}} \Omega^{\{k\}} \mathbbm{1}^{\slow} \otimes (A^{\fastfast} c^{\fast \times k}) = Lc^{\slow} \otimes \mathbbm{1}^{\fast} + \Delta c^{\slow} \otimes c^{\fast},
 \end{align*}
 which are equivalent to the conditions \eqref{eq:internalcons}.
 \end{proof}

 \subsubsection{IMEX-MRI-GARK Order Conditions}

 Due to the structure of the IMEX-MRI-GARK method \eqref{mriimexs}, many of the GARK order conditions are automatically satisfied.  As discussed in \cite{Sandu2019}, since $\Abf^{\islowislow} = A^{\islowislow}$, $\Abf^{\esloweslow} = A^{\esloweslow}$, $\bbf^{\islow} = b^{\islow}$, $\bbf^{\eslow} = b^{\eslow}$, $\cbf^{\islow} = c^{\slow}$, and $\cbf^{\eslow} = c^{\slow}$ from \eqref{eq:AslowGARK} and \eqref{eq:bislow}-\eqref{eq:ceslow}, and since our base IMEX-ARK method has order $q$, then all of the GARK order conditions up to order $q$ corresponding to only the ``slow'' components (and their couplings) will be satisfied.  Similarly, since `infinitesimal' methods assume that the fast component is solved exactly (or at least using an approximation of order $\geq q$), then the ``fast''  GARK order $q$ conditions will similarly be satisfied.  Additionally as discussed in \cite{sanduGeneralizedStructureApproachAdditive2015}, if all component tables have order at least two, then an IMEX-MRI-GARK method \eqref{mriimexs} that satisfies the internal consistency conditions from Theorem \ref{thm:internal_consistency} will be at least second order accurate.  Therefore, in this section we focus on only the remaining coupling conditions between the fast and slow components (both implicit and explicit) for orders three and four.


 We make use of the following simplifying conditions as listed in Lemma 3.8 of \cite{Sandu2019}, reproduced here in matrix form, taking into account the structure of our slow base IMEX-ARK method:
\begin{align}
    \label{eq:Asfcf} \Abf^{\slowfast} \cbf^{\fast} &= \frac{1}{2}c^{\slow \times 2},\\
    \label{eq:biasf} \bbf^{\islow T}\Abf^{\slowfast} &= \Big((\Delta c^{\slow} \times (Db^{\islow})) \otimes b^{\fast }\Big)^T,\\
    \label{eq:beasf} \bbf^{\eslow T}\Abf^{\slowfast} &= \Big((\Delta c^{\slow} \times (Db^{\eslow})) \otimes b^{\fast }\Big)^T,\\
    \label{eq:bfafi} \bbf^{\fast T}\Abf^{\fastislow} &= \Delta c^{\slow T} \mathcal{A}^{\{I,\zeta\}},\\
    \label{eq:bfafe} \bbf^{\fast T}\Abf^{\fasteslow} &= \Delta c^{\slow T} \mathcal{A}^{\{E,\zeta\}},\\
    \label{eq:Afics} \Abf^{\fastislow} \cbf^{\slow} &= \Bigg((LA^{\islowislow}) \otimes \mathbbm{1}^{\fast} + \sum_{k=0}^{k_{\max}} \Gamma^{\{k\}}  \otimes \Big(A^{\fastfast} c^{\fast \times k} \Big)\Bigg)c^{\slow},\\
    \label{eq:Afecs} \Abf^{\fasteslow} \cbf^{\slow} &= \Bigg((LA^{\esloweslow}) \otimes \mathbbm{1}^{\fast} + \sum_{k=0}^{k_{\max}} \Omega^{\{k\}}  \otimes \Big(A^{\fastfast} c^{\fast \times k} \Big)\Bigg)c^{\slow},\\
    \notag \text{and}\qquad\qquad&\\
    \label{eq:Affcf} \Abf^{\fastfast} \cbf^{\fast} &= \frac{1}{2} (Lc^{\slow})^{\times 2} \otimes \mathbbm{1}^{\fast} + \Big((Lc^{\slow}) \times \Delta c^{\slow}\Big) \otimes c^{\fast}\\
    & \qquad + \Delta c^{\slow \times 2} \otimes \Big(A^{\fastfast} c^{\fast} \Big)\notag,
\end{align}
where we use the notation $a\times b$ to indicate element-wise multiplication of two vectors, and where we define
\begin{align}
   \label{eq:Azeta}
   \mathcal{A}^{\{I,\zeta\}} &= LA^{\islowislow} + \sum_{k=0}^{k_{\max}} \zeta_{k} \Gamma^{\{k\}}, &
   \mathcal{A}^{\{E,\zeta\}} &= LA^{\esloweslow} + \sum_{k=0}^{k_{\max}} \zeta_{k} \Omega^{\{k\}},\\
   L_{i,j} &= \delta_{i,j+1}, & D_{i,j} &= \begin{cases} 1, & j \geq i, \\
     0, & \text{otherwise},
   \end{cases} \notag
\end{align}
and
\begin{equation}
\label{eq:zeta_k}
     \zeta_k = b^{\fast T} A^{\fastfast} c^{\fast \times k}.
\end{equation}



\begin{theorem}[Third order conditions]
\label{thm:order3}
  An internally consistent IMEX-MRI-GARK method \eqref{mriimexs} has order three iff the base IMEX-ARK method has order at least three, and the coupling conditions
 \begin{equation}
    \label{eq:order3}
    \Delta c^{\slow T} \mathcal{A}^{\{I,\zeta\}} c^{\slow} = \frac{1}{6} \qquad\text{and}\qquad
    \Delta c^{\slow T} \mathcal{A}^{\{E,\zeta\}} c^{\slow} = \frac{1}{6}
 \end{equation}
 hold, where $\mathcal{A}^{\{I,\zeta\}}$ and $\mathcal{A}^{\{E,\zeta\}}$ are defined in equation \eqref{eq:Azeta}.
 \end{theorem}

\begin{proof}
Using \eqref{eq:Asfcf}, we have that
\begin{equation*}
  \bbf^{\{\sigma\} T} \Abf^{\slowfast} \cbf^{\fast} = \frac12 b^{\slow T} c^{\slow \times 2} = \frac{1}{2}\Big(\frac{1}{3}\Big)
\end{equation*}
for $\sigma\in\{I,E\}$, and thus two of the third order GARK conditions are automatically satisfied.  Similarly, from \eqref{eq:bfafi} and \eqref{eq:bfafe} we have
\begin{align*}
   \bbf^{\fast T}\Abf^{\{F,\sigma\}}\cbf^{\slow} = \Delta c^{\slow T} \mathcal{A}^{\{\sigma,\zeta\}} c^{\slow},
\end{align*}
which result in the conditions \eqref{eq:order3}.
\end{proof}

\begin{theorem}[Fourth order conditions]
\label{thm:order4}
  An IMEX-MRI-GARK method \eqref{mriimexs} that satisfies Theorem \ref{thm:order3} has order four iff the base IMEX-ARK method has order at least four, and the following coupling conditions hold for $\sigma,\nu\in\{I,E\}$:
 \begin{subequations}
 \label{eq:order4}
 \begin{align}
    \label{eq:4a} \Big(\Delta c^{\slow} \times L c^{\slow}\Big)^T \mathcal{A}^{\{\sigma,\zeta\}} c^{\slow} + \Big(\Delta c^{\slow \times 2}\Big)^T \mathcal{A}^{\{\sigma,\beta\}} c^{\slow} &= \frac{1}{8},\\
    \label{eq:4b} \Delta c^{\slow T}\mathcal{A}^{\{\sigma,\zeta\}} c^{\slow \times 2} &= \frac{1}{12},\\
    \label{eq:4c} \Big(\Delta c^{\slow}\times (Db^{\{\sigma\})} \Big)^T\mathcal{A}^{\{\nu,\zeta\}} c^{\slow} &= \frac{1}{24},\\
    \label{eq:4d} \Big(\Delta c^{\slow \times 2}\Big)^T \mathcal{A}^{\{\sigma,\xi\}} c^{\slow} + \Delta c^{\slow T} L \Delta C^{\slow} \mathcal{A}^{\{\sigma,\zeta\}} c^{\slow} &= \frac{1}{24}, \quad\text{and}\\
    \label{eq:4e} \Delta c^{\slow T} \mathcal{A}^{\{\sigma,\zeta\}} A^{\{\nu,\nu\}} c^{\slow} &= \frac{1}{24},
 \end{align}
 \end{subequations}
 where we have defined the auxiliary variables
 \begin{align}
   \label{eq:AIbeta}
   \mathcal{A}^{\{I,\beta\}} &:= \frac{1}{2}LA^{\islowislow} +\sum_{k=0}^{k_{\max}} \beta_{k} \Gamma^{\{k\}},\\
   \label{eq:AEbeta}
   \mathcal{A}^{\{E,\beta\}} &:= \frac{1}{2}LA^{\esloweslow} +\sum_{k=0}^{k_{\max}} \beta_{k} \Omega^{\{k\}},\\
   \label{eq:AIxi}
   \mathcal{A}^{\{I,\xi\}} &:= \frac{1}{2}LA^{\islowislow} +\sum_{k=0}^{k_{\max}} \xi_{k} \Gamma^{\{k\}},\\
   \label{eq:AExi}
   \mathcal{A}^{\{E,\xi\}} &:= \frac{1}{2}LA^{\esloweslow} +\sum_{k=0}^{k_{\max}} \xi_{k} \Omega^{\{k\}},\\
   \label{eq:beta_k}
   \beta_k &:= (b^{\fast} \times c^{\fast})^T A^{\fastfast} c^{\fast \times k}, \quad\text{and}\\
   \label{eq:xi_k}
   \xi_k &:= b^{\fast T} A^{\fastfast} A^{\fastfast} c^{\fast \times k}.
 \end{align}
 \end{theorem}

\begin{proof}
Since the GARK representation of our IMEX-MRI-GARK method is internally consistent, there are 26 coupling conditions of order 4.  Of these, ten are automatically satisfied due the IMEX-MRI-GARK method structure and our assumed accuracy of the base IMEX-ARK method: for $\sigma,\nu\in\{I,E\}$,
\begin{subequations}
\label{eq:order4_auto}
\begin{align}
  \label{eq:bscsAsfcf}
  \Big(\bbf^{\{\sigma \}} \times \cbf^{\slow } \Big)^{T} \Abf^{\slowfast} \cbf^{\fast} &=
  \frac18,\\
\label{eq:bsAssAsfcf}
    \bbf^{\{\sigma\} T}\Abf^{\{\nu,\nu\}}\Abf^{\slowfast}\cbf^{\fast} &=
    \frac{1}{24},\\
  \label{eq:bsAsfcf2}
  \bbf^{\{\sigma\} T}\Abf^{\slowfast} \cbf^{\fast \times 2} &=
    \frac{1}{12}, \quad\text{and}\\
  \label{eq:bsAsfAffcf}
    \bbf^{\{\sigma\} T}\Abf^{\slowfast}\Abf^{\fastfast}\cbf^{\fast} &=
    \frac{1}{24}.
\end{align}
\end{subequations}
The remaining 16 coupling conditions are
\begin{subequations}
\label{eq:order4_other}
\begin{align}
  \label{eq:bfcfAfscs}
    \Big(\bbf^{\fast } \times \cbf^{\fast } \Big)^{T} \Abf^{\{F,\sigma\}} \cbf^{\slow} &= \frac18,\\
  \label{eq:bfAfscs2}
    \bbf^{\fast T} \Abf^{\{F,\sigma\}} c^{\slow \times 2} &= \frac{1}{12},\\
  \label{eq:bsAsfAfscs}
    \bbf^{\{\sigma\} T} \Abf^{\slowfast} \Abf^{\{F,\nu\}}\cbf^{\slow} &= \frac{1}{24},\\
  \label{eq:bfAffAfscs}
    \bbf^{\fast T} \Abf^{\fastfast} \Abf^{\{F,\sigma\}} \cbf^{\slow} &= \frac{1}{24},\\
  \label{eq:bfAfsAsscs}
    \bbf^{\fast T} \Abf^{\{F,\sigma\}} \Abf^{\{\nu,\nu\}}\cbf^{\slow} &= \frac{1}{24},\quad\text{and}\\
  \label{eq:bfAfsAsfcf}
    \bbf^{\fast T}\Abf^{\{F,\sigma\}}\Abf^{\slowfast}\cbf^{\fast} &= \frac{1}{24},
\end{align}
\end{subequations}
where again $\sigma,\nu\in\{I,E\}$.

We first prove the automatically-satisfied conditions \eqref{eq:order4_auto}.
Using \eqref{eq:Asfcf} and our assumption that the base IMEX-ARK method is order four,
\begin{equation*}
  \Big(\bbf^{\{\sigma \}} \times \cbf^{\slow } \Big)^{T} \Abf^{\slowfast} \cbf^{\fast} = \frac12 b^{\{\sigma\} T} c^{\slow \times 3}  = \frac{1}{2}\Big(\frac{1}{4} \Big)
\end{equation*}
and
\begin{equation*}
    \bbf^{\{\sigma\} T}\Abf^{\{\nu,\nu\}}\Abf^{\slowfast}\cbf^{\fast} = \frac12 b^{\{\sigma\} T}A^{\{\nu,\nu\}}c^{\slow \times 2} = \frac{1}{2} \Big(\frac{1}{12} \Big),
\end{equation*}
for $\sigma,\nu\in\{I,E\}$, and hence \eqref{eq:bscsAsfcf} and \eqref{eq:bsAssAsfcf} are satisfied.
Using the definition of $\cbf^{\fast}$ from \eqref{eq:cfast}, the simplifying formulas \eqref{eq:biasf}-\eqref{eq:beasf}, and our assumptions that $c^{\slow}_1=0$, the fast method is at least third order, and the IMEX-ARK method is at least fourth order, we have for $\sigma\in\{I,E\}$:
\begin{align*}
  &\bbf^{\{\sigma\} T}\Abf^{\slowfast} \cbf^{\fast \times 2} \\
  &= \left((\Delta c^{\slow} \times (Db^{\{\sigma\}})) \otimes b^{\fast }\right)^T  \left(Lc^{\slow} \otimes \mathbbm{1}^{
  \fast}+ \Delta c^{\slow} \otimes c^{\fast} \right)^{\times 2} \\
  &= \left(\Delta c^{\slow} \times (Db^{\{\sigma\})}\right)^T \left((Lc^{\slow})^{\times 2} + (Lc^{\slow} \times \Delta c^{\slow})  + \frac{1}{3}\Delta c^{\slow \times 2} \right)\\
  &= \left(D b^{\{\sigma\}} \right)^T \left((Lc^{\slow})^{\times 2}\times \Delta c^{\slow} + Lc^{\slow} \times \Delta c^{\slow \times 2}  + \frac{1}{3}\Delta c^{\slow \times 3} \right)\\
  &= \frac{1}{3} \left(D b^{\{\sigma\}} \Big)^T \Big(c^{\slow \times 3} - (L c^{\slow})^{\times 3} \right)\\
  &= \frac{1}{3} \sum_{i=2}^{\sslow}\left(\sum_{l=i}^{\sslow} b_l^{\{\sigma\}}\right) \left(c^{\slow \times 3}_i - c^{\slow \times 3}_{i-1}\right) = \frac{1}{3} b^{\{\sigma\} T}c^{\slow \times 3} = \frac{1}{3}\left(\frac{1}{4} \right),
\end{align*}
which proves the coupling conditions \eqref{eq:bsAsfcf2}.
Using the simplifying formulas \eqref{eq:biasf}, \eqref{eq:beasf} and \eqref{eq:Affcf}, and the same assumptions as in the previous step, for $\sigma\in\{I,E\}$ we have
\begin{align*}
  &\bbf^{\{\sigma\} T} \Abf^{\slowfast} \Abf^{\fastfast} \cbf^{\fast} \\
  &= \left((\Delta c^{\slow} \times (Db^{\{\sigma\}})) \otimes b^{\fast }\right)^T \Bigg(\frac12 (Lc^{\slow})^{\times 2} \otimes \mathbbm{1}^{\fast} \\
  &\quad + \left((Lc^{\slow}) \times \Delta c^{\slow}\right) \otimes c^{\fast} + \Delta c^{\slow \times 2} \otimes \left(A^{\fastfast} c^{\fast} \right)\Bigg)\\
  &= \left(\Delta c^{\slow} \times (Db^{\{\sigma\})}\right)^T\left(\frac12 (Lc^{\slow})^{\times 2}  + \frac12 (Lc^{\slow}) \times \Delta c^{\slow} + \frac16 \Delta c^{\slow \times 2} \right)\\
  &= \frac16 \left(Db^{\{\sigma\}}\right)^T \left(c^{\slow \times 3}-(Lc^{\slow})^{\times 3} \right)
  = \frac16 b^{\{\sigma\} T}c^{\slow \times 3}
  = \frac16 \left(\frac14\right),
\end{align*}
and thus the coupling conditions \eqref{eq:bsAsfAffcf} are automatically satisfied as well.

We now examine the 16 remaining fourth-order GARK conditions \eqref{eq:order4_other}.  Starting with \eqref{eq:bfcfAfscs}, we use the definitions \eqref{eq:bfast} and \eqref{eq:cfast}, the simplifying formulas \eqref{eq:Afics}-\eqref{eq:Afecs}, and that the fast method is at least second order to obtain:
\begin{align*}
  \frac18
  &= \left(\bbf^{\fast } \times \cbf^{\fast } \right)^{T} \Abf^{\fastislow} \cbf^{\slow} \\
  & =  \left(\left(\Delta c^{\slow} \otimes b^{\fast} \right) \times \left(Lc^{\slow} \otimes \mathbbm{1}^{\fast} + \Delta c^{\slow} \otimes c^{\fast} \right) \right)^{T}\\
  &\qquad \left(L A^{\islowislow} \otimes \mathbbm{1}^{\fast} + \sum_{k=0}^{k_{\max}} \Gamma^{\{k\}}  \otimes \left(A^{\fastfast} c^{\fast \times k} \right)\right)c^{\slow}\\
  & = \left(\Delta c^{\slow} \times L c^{\slow} \right)^{T} \mathcal{A}^{\{I,\zeta\}} c^{\slow} + \left(\Delta c^{\slow \times 2} \right)^{T} \mathcal{A}^{\{I,\beta\}} c^{\slow}.
\end{align*}
A similar argument gives
\[
  \frac18 = \left(\Delta c^{\slow} \times L c^{\slow} \right)^{T} \mathcal{A}^{\{E,\zeta\}} c^{\slow} + \left(\Delta c^{\slow \times 2} \right)^{T} \mathcal{A}^{\{E,\beta\}} c^{\slow},
\]
which establishes the conditions \eqref{eq:4a}.
Using the simplifying formulas \eqref{eq:bfafi}-\eqref{eq:bfafe}, the order conditions \eqref{eq:bfAfscs2} become
\begin{equation*}
  \frac{1}{12} = \bbf^{\fast T} \Abf^{\{F,\sigma\}} \cbf^{\slow \times 2} = \Delta c^{\slow T} \mathcal{A}^{\{\sigma,\zeta\}} c^{\slow \times 2}
\end{equation*}
for $\sigma\in\{I,E\}$, which are equivalent to the conditions \eqref{eq:4b}.
For the order conditions \eqref{eq:bsAsfAfscs}, we use simplifying formulas \eqref{eq:biasf}-\eqref{eq:beasf} and \eqref{eq:Afics} to obtain for $\sigma\in\{I,E\}$:
\begin{align*}
  &\frac{1}{24} = \bbf^{\{\sigma\} T} \Abf^{\slowfast} \Abf^{\{F,I\}}\cbf^{\slow} \\
  &= \Big((\Delta c^{\slow} \times (Db^{\{\sigma\}})) \otimes b^{\fast} \Big)^T \Big(LA^{\islowislow} \otimes \mathbbm{1}^{\fast} + \sum_{k=0}^{k_{\max}} \Gamma^{\{k\}} \otimes (A^{\fastfast}c^{\fast \times k} \Big)c^{\slow}\\
  &= (\Delta c^{\slow} \times (Db^{\{\sigma\}}))^{T} \mathcal{A}^{\{I,\zeta\}} c^{\slow}.
\end{align*}
Similarly using the simplifying formulas \eqref{eq:biasf}-\eqref{eq:beasf} and \eqref{eq:Afecs}, we have
\[
  \frac{1}{24} = (\Delta c^{\slow} \times (Db^{\{\sigma\}}))^{T} \mathcal{A}^{\{E,\zeta\}} c^{\slow},
\]
resulting in the conditions \eqref{eq:4c}.
We use the definitions \eqref{eq:bfast} and \eqref{eq:Afastfast}, and the simplifying formula \eqref{eq:Afics} to convert the order condition \eqref{eq:bfAffAfscs} for $\sigma=I$:
\begin{align*}
  \frac{1}{24} &= \bbf^{\fast T} \Abf^{\fastfast} \Abf^{\fastislow} \cbf^{\slow} \\
  &= \left(\Delta c^{\slow} \otimes b^{\fast} \right)^{T} \left(\text{diag}\left(\Delta c^{\slow} \right) \otimes A^{\fastfast} + L \Delta C^{\slow} \otimes \mathbbm{1}^{\fast} b^{\fast T}\right)\\
  & \qquad \left(LA^{\islowislow} \otimes \mathbbm{1}^{\fast} + \sum_{k=0}^{k_{\max}} \Gamma^{\{k\}}  \otimes \left(A^{\fastfast} c^{\fast \times k} \right)\right)c^{\slow}\\
  &= \left(\left(\Delta c^{\slow \times 2}\right)^T \otimes (b^{\fast T} A^{\fastfast})  + \Delta c^{\slow T} L \Delta C^{\slow} \otimes b^{\fast T} \right)\\
  & \qquad \left(LA^{\islowislow} \otimes \mathbbm{1}^{\fast} + \sum_{k=0}^{k_{\max}} \Gamma^{\{k\}}  \otimes \left(A^{\fastfast} c^{\fast \times k} \right)\right)c^{\slow}\\
  &= \left(\Delta c^{\slow \times 2}\right)^T \mathcal{A}^{\{I,\xi\}} c^{\slow} + \Delta c^{\slow T} L \Delta C^{\slow} \mathcal{A}^{\{I,\zeta\}} c^{\slow}.
\end{align*}
Similarly, the simplifying formula \eqref{eq:Afecs} converts \eqref{eq:bfAffAfscs} for $\sigma=E$ to
\[
  \frac{1}{24} = \left(\Delta c^{\slow \times 2}\right)^T \mathcal{A}^{\{E,\xi\}} c^{\slow} + \Delta c^{\slow T} L \Delta C^{\slow} \mathcal{A}^{\{E,\zeta\}} c^{\slow},
\]
which establishes the conditions \eqref{eq:4d}.
Using the simplifying formulas \eqref{eq:bfafi} and \eqref{eq:bfafe}, the order conditions \eqref{eq:bfAfsAsscs} become for $\sigma,\nu\in\{I,E\}$:
\begin{align*}
  \frac{1}{24} = \bbf^{\fast T} \Abf^{\{F,\sigma\}} \Abf^{\{\nu,\nu\}}\cbf^{\slow} &= \Delta c^{\slow T} \mathcal{A}^{\{\sigma,\zeta\}} A^{\{\nu,\nu\}} c^{\slow},
\end{align*}
which are the coupling conditions \eqref{eq:4e}.
The final order conditions, \eqref{eq:bfAfsAsfcf}, may be simplified using formulas \eqref{eq:Asfcf} and \eqref{eq:bfafi}-\eqref{eq:bfafe} for $\sigma\in\{I,E\}$:
\begin{align*}
  \frac{1}{24} = \bbf^{\fast T} \Abf^{\{F,\sigma\}} \Abf^{\slowfast} \cbf^{\fast} = \frac12 \Delta c^{\slow} \mathcal{A}^{\{\sigma,\zeta\}} c^{\slow \times 2},
\end{align*}
which are equivalent to the coupling conditions \eqref{eq:4b}.
\end{proof}
\begin{remark}
For many IMEX-ARK methods the coefficients are chosen so that $\bbf^{\eslow} = \bbf^{\islow}$ to reduce the number of order conditions that must be satisfied.  Similarly, when $\bbf^{\eslow} = \bbf^{\islow}$ many of the 3-component GARK order conditions (on which IMEX-MRI-GARK methods rely) are duplicated.  One could then wonder whether the assumption $\bbf^{\eslow} = \bbf^{\islow}$ would significantly reduce the number of order conditions required to derive IMEX-MRI-GARK methods.  This is not in fact the case, since the large majority of these duplicated GARK order conditions are already automatically satisfied in \eqref{eq:order4_auto} due to the IMEX-MRI-GARK structure and our assumptions on the order of the underlying IMEX-ARK method.  Of the remaining 16 GARK order conditions in \eqref{eq:order4_other} that are not automatically satisfied, only the conditions \eqref{eq:bsAsfAfscs} (that correspond with the IMEX-MRI-GARK condition \eqref{eq:4c}) benefit from an assumption that $\bbf^{\eslow} = \bbf^{\islow}$, causing those 4 conditions to simplify to 2.  Thus although all of the IMEX-MRI-GARK methods presented later in Section \ref{sec:example_tables} are derived from IMEX-ARK methods satisfying $\bbf^{\eslow} = \bbf^{\islow}$, this should by no means be considered as a requirement when deriving new IMEX-MRI-GARK methods.
\end{remark}

\section{Linear Stability}
\label{section3}
There is no standard theoretical framework for analyzing linear stability of methods for additive problems (of \emph{either} form \eqref{eq:usualform} or \eqref{eq:probode}).  Thus although it relies on an assumption that the Jacobians with respect to $y$ of $\fislow$, $\feslow$ and $\ffast$ are simultaneously diagonalizable, similar to \cite{Sandu2019} we analyze linear stability on an additive scalar test problem:
\begin{equation}
\label{eq:scalartest}
    y' = \lambda^{\fast}y + \lambda^{\eslow}y + \lambda^{\islow}y,\quad t\ge 0,\quad y(0) = 1,
\end{equation}
where each of $\lambda^{\fast}, \lambda^{\eslow} , \lambda^{\islow} \in \mathbb{C}^{-}$, and we define $z^{\fast} := H\lambda^{\fast}$, $z^{\eslow} := H \lambda^{\eslow}$, and $z^{\islow} := H\lambda^{\islow}$. Applying the IMEX-MRI-GARK method \eqref{mriimexs} to the scalar model problem \eqref{eq:scalartest}, the modified fast IVP for each slow stage  $i = 2,\ldots, \sslow$ becomes:
\begin{align*}
    v' &= \Delta c_i^{\slow} \lambda^{\fast}v + \lambda^{\eslow} \sum_{j = 1}^{i-1} \omega_{i,j}\!\left(\frac{\theta}{H} \right) Y_j^{\slow} + \lambda^{\islow} \sum_{j = 1}^{i} \gamma_{i,j}\!\left(\frac{\theta}{H}\right) Y_j^{\slow}\\
    &= \Delta c_i^{\slow} \lambda^{\fast}v + \lambda^{\eslow} \sum_{j = 1}^{i-1}\sum_{k=0}^{k_{\max}}  \omega_{i,j}^{\{k\}}\frac{\theta^k}{H^k} Y_j^{\slow} +
    \lambda^{\islow} \sum_{j = 1}^{i}
    \sum_{k=0}^{k_{\max}} \gamma_{i,j}^{\{k\}}\frac{\theta^k}{H^k} Y_j^{\slow},
\end{align*}
for $\theta\in[0,H]$, with initial condition $v(0) = Y_{i-1}^{\slow}$.
We solve for the updated slow stage $Y_{i}^{\slow}:=v(H)$ analytically using the variation of constants formula:
\begin{align}
  \label{eq:yplus1stab}
  Y_{i}^{\slow} &=
  e^{\Delta c_i^{\slow} z^{\fast}} Y_{i-1}^{\slow} + z^{\eslow} \sum_{j=1}^{i-1} \sum_{k=0}^{k_{\max}} \omega_{i,j}^{\{k\}} \Bigg( \int_{0}^{1} e^{\Delta c_i^{\slow} z^{\fast} (1-t)} t^{k} \mathrm dt \Bigg) Y_j^{\slow} \\
  &\qquad  + z^{\islow} \sum_{j=1}^{i} \sum_{k=0}^{k_{\max}} \gamma_{i,j}^{\{k\}} \Bigg( \int_{0}^{1} e^{\Delta c_i^{\slow} z^{\fast} (1-t)} t^{k} \mathrm dt \Bigg) Y_j^{\slow} \notag \\
  &= \varphi_0\!\left(\Delta c_i^{\slow} z^{\fast} \right) Y_{i-1}^{\slow} + z^{\eslow} \sum_{j=1}^{i-1} \eta_{i,j}(z^{\fast}) Y_j^{\slow} + z^{\islow} \sum_{j=1}^{i} \mu_{i,j}(z^{\fast}) Y_j^{\slow}, \notag
\end{align}
where $\eta$ and $\mu$ depend on the fast variable:
\begin{align*}
  \eta_{i,j}\!\left(z^{\fast}\right) &= \sum_{k=0}^{k_{\max}} \omega_{i,j}^{\{k\}} \varphi_{k+1}\!\left(\Delta c_i^{\slow} z^{\fast}\right)\\
  \mu_{i,j}\!\left(z^{\fast}\right) &= \sum_{k=0}^{k_{\max}} \gamma_{i,j}^{\{k\}} \varphi_{k+1}\!\left(\Delta c_i^{\slow} z^{\fast}\right),
\end{align*}
and the family of analytical functions $\{\varphi_k\}$ are defined as in \cite{Sandu2019},
\begin{equation*}
    \varphi_0(z) = e^{z}, \quad \varphi_{k}(z) = \int_{0}^{1} e^{z(1-t)} t^{k-1} \mathrm dt,\quad k \geq 1,
\end{equation*}
or recursively as
\begin{equation*}
    \varphi_{k+1}(z) = \frac{k\, \varphi_k(z) - 1}{z}, \quad k \geq 1.
\end{equation*}
Concatenating $Y=\begin{bmatrix} Y_{1}^{\slow T} & \cdots & Y_{\sslow}^{\slow T}\end{bmatrix}^T$, we can write \eqref{eq:yplus1stab} in matrix form as
\begin{align*}
  Y &= \text{diag} \Big(\varphi_0\big(\Delta c^{\slow} z^{\fast} \big) \Big)LY + \varphi_0\big(\Delta c_1 z^{\fast}\big)y_n e_1 \\
  &\qquad + z^{\eslow} \eta \big(z^{\fast}\big)Y + z^{\islow} \mu \big(z^{\fast}\big) Y \notag\\
  &= \Bigg(I - \text{diag} \Big(\varphi_0\big(\Delta c^{\slow} z^{\fast} \big) \Big)L - z^{\eslow} \eta \big(z^{\fast}\big) - z^{\islow} \mu \big(z^{\fast}\big)  \Bigg)^{-1}y_n e_1,
\end{align*}
where
\begin{align*}
  \eta \big(z^{\fast}\big) &= \sum_{k=0}^{k_{\max}} \text{diag} \Big(\varphi_{k+1}\big(\Delta c^{\slow} z^{\fast} \big) \Big) \Omega^{\{k\}} \quad\text{and}\\
  \mu \big(z^{\fast}\big) &= \sum_{k=0}^{k_{\max}} \text{diag} \Big(\varphi_{k+1}\big(\Delta c^{\slow} z^{\fast} \big) \Big) \Gamma^{\{k\}}.
\end{align*}
Thus the linear stability function for IMEX-MRI-GARK on the problem \eqref{eq:scalartest} becomes
\begin{align}
\label{eq:stability_function}
    &R\big(z^{\fast},z^{\eslow}, z^{\islow}\big) \\
    &:= e_{\sslow}^{T} \Bigg(I - \text{diag} \Big(\varphi_0\big(\Delta c^{\slow} z^{\fast} \big) \Big)L - z^{\eslow} \eta \big(z^{\fast}\big)  - z^{\islow} \mu \big(z^{\fast}\big)  \Bigg)^{-1} e_1. \notag
\end{align}
Following a similar definition as in \cite{Zharovsky2015Class}, we define the joint stability for the slow, nonstiff region as:
\begin{equation*}
    \mathcal{J}_{\alpha,\beta} \coloneqq \left\{z^{\eslow} \in \mathbb{C}^{-} \;:\; |R(z^{\fast},z^{\eslow},z^{\islow})|\leq 1,\; \forall z^{\fast} \in \mathcal{S}_{\alpha}^{\fast},\; \forall z^{\islow} \in \mathcal{S}_{\beta}^{\islow} \right\}
\end{equation*}
where
    $\mathcal{S}_{\alpha}^{\sigma} := \left\{z^{\sigma} \in \mathbb{C}^{-} \;:\; |\arg (z^{\sigma})- \pi| \leq \alpha \right\}$.
Since such stability regions are not widespread in the literature, we highlight the role of each component, before plotting these for candidate IMEX-MRI-GARK methods in the next section.  $\mathcal{J}_{\alpha,\beta}$ provides a plot of the stability region for the slow explicit component only, under assumptions that (a) $z^{\islow}$ can range throughout an entire infinitely long sector $\mathcal{S}_{\alpha}^{\islow}$ in the complex left half plane, and (b) $z^{\fast}$ can range throughout another [infinite] sector $\mathcal{S}_{\beta}^{\fast}$ in $\mathbb{C}^{-}$.  These sectors both include the entire negative real axis, as well as a swath of values with angle at most $\alpha$ or $\beta$ above and below this axis, respectively.  As such, one should expect the joint stability region $\mathcal{J}_{\alpha,\beta}$ to be significantly smaller than the standard stability region for just the slow explicit table $(A^{\eslow},b^{\eslow},c^{\eslow})$, and to shrink in size as both $\alpha,\beta$ increase.  Furthermore, we note that this notion of a joint stability region is artificially restrictive, since in practice the functions $\fislow$ and $\ffast$ will not be \emph{infinitely} stiffer than $\feslow$.


\section{Example IMEX-MRI-GARK Methods}
\label{sec:example_tables}

While our focus in this paper is on the underlying theory regarding IMEX-MRI-GARK methods of the form \eqref{def:IMEX-MRI-method}, in this section we discuss how IMEX-MRI-GARK methods may be constructed, and provide methods of orders 3 and 4 to use in demonstrating our numerical results in Section \ref{sec:numerical_results}.

\subsection{Third-order Methods}
\label{sec:IMEX-MRI3}
We create two third order IMEX-MRI-GARK methods, both based on the `(3,4,3)' IMEX-ARK method from Section 2.7 of \cite{ascherImplicitexplicitRungeKuttaMethods1997},
\begin{center}
\begin{tabular}{c|cccc|cccc}
                   $0$ &         $0$ &       $0$ &      $0$ & $0$ & $0$ &      $0$ & $0$ & $0$ \\
              $\eta$ &    $\eta$ &       $0$ &      $0$ & $0$ & $0$ & $\eta$ & $0$ & $0$ \\
  $\frac{1+\eta}{2}$ &   $a_{3,1}$ & $a_{3,2}$ &      $0$ & $0$ & $0$ & $\frac{1-\eta}{2}$ & $\eta$ & $0$ \\
                   $1$ & $1-2\alpha$ &  $\alpha$ & $\alpha$ & $0$ & $0$ & $b_2$ & $b_3$ & $\eta$ \\
  \hline
                   $1$ & $0$ & $b_2$ & $b_3$ & $\eta$ & $0$ & $b_2$ & $b_3$ & $\eta$
\end{tabular}
\end{center}
where
\begin{align*}
  \eta &= 0.4358665215084589994160194511935568425293,\\
  \alpha &= 0.5529291480359398193611887297385924764949,\\
  a_{3,2} &= \left(-\frac{15}{4} + 15\eta - \frac{21}{4}\eta^2\right)\alpha + 4 - \frac{25}{2}\eta + \frac92\eta^2,\\
  a_{3,1} &= \left(\frac{15}{4} - 15\eta + \frac{21}{4}\eta^2\right)\alpha - \frac72 + 13\eta - \frac92\eta^2,\\
  b_2 &= -\frac32\eta^2 + 4\eta - \frac14,\\
  b_3 &= \frac32\eta^2 - 5\eta + \frac54.
\end{align*}
As the explicit portion of this pair is not `stiffly accurate' we pad the tables as discussed in Section \ref{section2}.  We then convert this to `solve-decoupled' form \cite{Sandu2019} by inserting additional rows and columns into the tables to ensure that any stage with a nonzero diagonal value in the slow implicit table is associated with $\Delta c_i=0$,
\vspace{1em}
\scriptsize
\begin{center}
\begin{tabular}{c|cccccccc|cccccccc}
                   0 &           0 & 0 &         0 & 0 &        0  &        0 & 0 &      0&    0 & 0 &                    0 & 0 &        0 & 0 &        0 & 0      \\
            $\eta$ &      $\eta$ & 0 &         0 & 0 &        0 &         0 & 0 &      0&  $\eta$ & 0 &                    0 & 0 &        0 & 0 &        0 & 0       \\
            $\eta$ &    $\eta$ & 0 &         0 & 0 &        0 &        0 & 0 &      0&    0 & 0 &             $\eta$ & 0 &        0 & 0 &        0 & 0       \\
$\frac{1+\eta}{2}$ &      $\Box$ & 0 &    $\Box$ & 0 &        0 &      0 & 0 &      0& $\Box$ & 0 &               $\Box$ & 0 &        0 & 0 &        0 & 0      \\
$\frac{1+\eta}{2}$ &   $a_{3,1}$ & 0 & $a_{3,2}$ & 0 &        0 & 0 &        0 &      0&    0 & 0 & $\frac{1-\eta}{2}$ & 0 & $\eta$ & 0 &        0 & 0        \\
                 $1$ &      $\Box$ & 0 &    $\Box$ & 0 &   $\Box$ & 0 &        0 &      0& $\Box$ & 0 &               $\Box$ & 0 &   $\Box$ & 0 &        0 & 0   \\
                 $1$ & $1-2\alpha$ & 0 &  $\alpha$ & 0 & $\alpha$ & 0 &        0 &      0&    0 & 0 &                $b_2$ & 0 &    $b_3$ & 0 & $\eta$ & 0     \\
                 $1$ &           0 & 0 &     $b_2$ & 0 &    $b_3$ & 0 & $\eta$ & 0 &         0 & 0 &                $b_2$ & 0 &    $b_3$ & 0 & $\eta$ & 0      \\
                  \hline
                 $1$ &           0 & 0 &     $b_2$ & 0 &    $b_3$ & 0 & $\eta$ & 0 &         0 & 0 &                $b_2$ & 0 &    $b_3$ & 0 & $\eta$ & 0
\end{tabular}
\end{center}
\normalsize
\vspace{1em}
where each entry in $A^{\esloweslow}$ and $A^{\islowislow}$ above labeled with $\Box$ need only be chosen to satisfy internal consistency for the ARK table.
We note that although the proposed IMEX-MRI-GARK methods \eqref{mriimexs} do not \emph{require} that the implicit portion of the IMEX-ARK table have this `solve-decoupled' pattern, we create tables with this structure due to their ease of implementation.  Specifically, if the corresponding IMEX-MRI-GARK method included a `solve-coupled' stage $i$ (i.e., both $\overline{\gamma}_{i,i} \ne 0$ and $\Delta c_i \ne 0$), then the stage solution $Y_i^{\slow}$ must \emph{both} define the fast IVP right-hand side \eqref{eq:b},
\begin{align*}
   v'(\theta) &= \Delta c_i^{\slow} \ffast\!\left(T_{i-1}+\Delta c_i^{\slow}\theta,\, v(\theta)\right) + \sum_{j=1}^{i-1} \left(\gamma_{i,j}\!\left(\tfrac{\theta}{H}\right)f^{\islow}_j + \omega_{i,j}\!\left(\tfrac{\theta}{H}\right)f^{\eslow}_j\right) \\
   &+ \gamma_{i,i}\!\left(\tfrac{\theta}{H}\right)f^{\islow}\!\left(t_n+c_i^{\slow}H, Y_i^{\slow}\right),
   \quad \theta \in [0,H],
\end{align*}
\emph{and} be the solution to this fast IVP, $Y_i^{\slow} = v(H)$.  Solve-decoupled methods, on the other hand, may be performed by alternating between standard implicit solves for each implicit stage, followed by fast evolution for non-implicit stages.  However, as noted in \cite{robertsImplicitMultirateGARK2019,robertsCoupledMultirateInfinitesimal2020}, while the solve-decoupled approach makes for easier implementation of MRI methods, it also results in methods with diminished stability.

The first IMEX-MRI-GARK method that we built from the table above is ``IMEX-MRI-GARK3a''. We simultaneously found the 10 $\Box$ values to complete the IMEX-ARK table, the 24 unknown $\Gamma^{\{0\}}$ coefficients and the 20 unknown $\Omega^{\{0\}}$ coefficients by solving the ARK consistency conditions \eqref{eq:AslowGARK}, the internal consistency conditions \eqref{eq:internalcons}, and the third order conditions \eqref{eq:thirdordercoupling}. Since this only constitutes 50 unique conditions that depend linearly on 54 unknown entries, the corresponding linear system of equations was under-determined.  For IMEX-MRI-GARK3a we used the particular solution returned by MATLAB (a minimum-norm least-squares solution). The resulting nonzero coefficients $\cbf^{\slow}$, $\Gamma^{\{0\}}$ and $\Omega^{\{0\}}$ are provided in Appendix \ref{sec:IMEX-MRI3a}.

Our second IMEX-MRI-GARK method, ``IMEX-MRI-GARK3b,'' was also constructed from this same base IMEX-ARK table.  Here, beginning with the IMEX-MRI-GARK3a particular solution above, we then used the four remaining free variables to maximize the extent of the joint stability region along the negative real-axis.  The nonzero coefficients $\cbf^{\slow}$, $\Gamma^{\{0\}}$ and $\Omega^{\{0\}}$ for the resulting method are given in Appendix \ref{sec:IMEX-MRI3b}.

\begin{remark}
An alternative approach for creating solve-decoupled third order IMEX-MRI-GARK methods is to take advantage of the free $\Box$ variables within the extended IMEX-ARK table, plus assumptions that $\overline{\Gamma} = \Gamma^{\{0\}}$ and $\overline{\Omega} = \Omega^{\{0\}}$.  Here, one may select the $\Box$ values to ensure that the IMEX-ARK is internally consistent and satisfies
\begin{equation}
\label{eq:thirdordercoupling}
    \Delta c^{\slow T}\Big(L + \tfrac{1}{2} E^{-1} \Big)A^{\esloweslow}c^{\slow} = \frac{1}{6}, \quad \Delta c^{\slow T}\Big(L + \tfrac{1}{2} E^{-1} \Big)A^{\islowislow}c^{\slow} = \frac{1}{6},
\end{equation}
as these are equivalent to the third order coupling conditions \eqref{eq:order3}, with equation \eqref{eq:AslowGARK} providing one-to-one correspondences between $A^{\islowislow}$ and $\Gamma^{\{0\}}$, and between $A^{\esloweslow}$ and $\Omega^{\{0\}}$.  We note that the conditions \eqref{eq:thirdordercoupling} each correspond to the previously-discovered third order condition for MIS methods introduced in \cite{knothImplicitExplicit1998}.
\end{remark}

\begin{figure}[htbp]
    \centering
    \subfigure[$\mathcal{J}_{10^o,\beta}$ for IMEX-MRI-GARK3a]{\includegraphics[width = .49\textwidth,trim={2.5cm 0 1.5cm 0},clip]{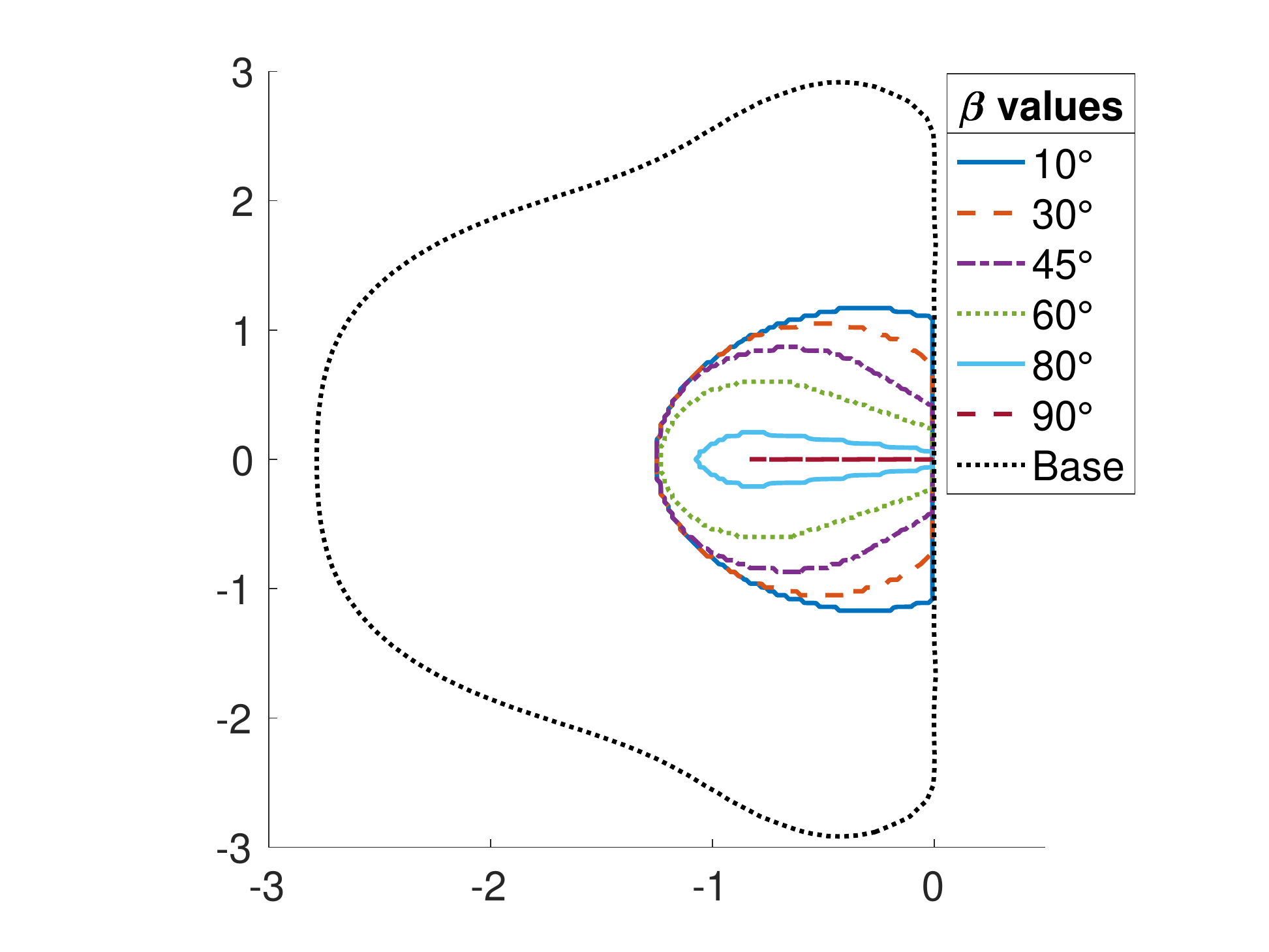}}
    \subfigure[$\mathcal{J}_{10^o,\beta}$ for IMEX-MRI-GARK3b]{\includegraphics[width = .49\textwidth,trim={2.5cm 0 1.5cm 0},clip]{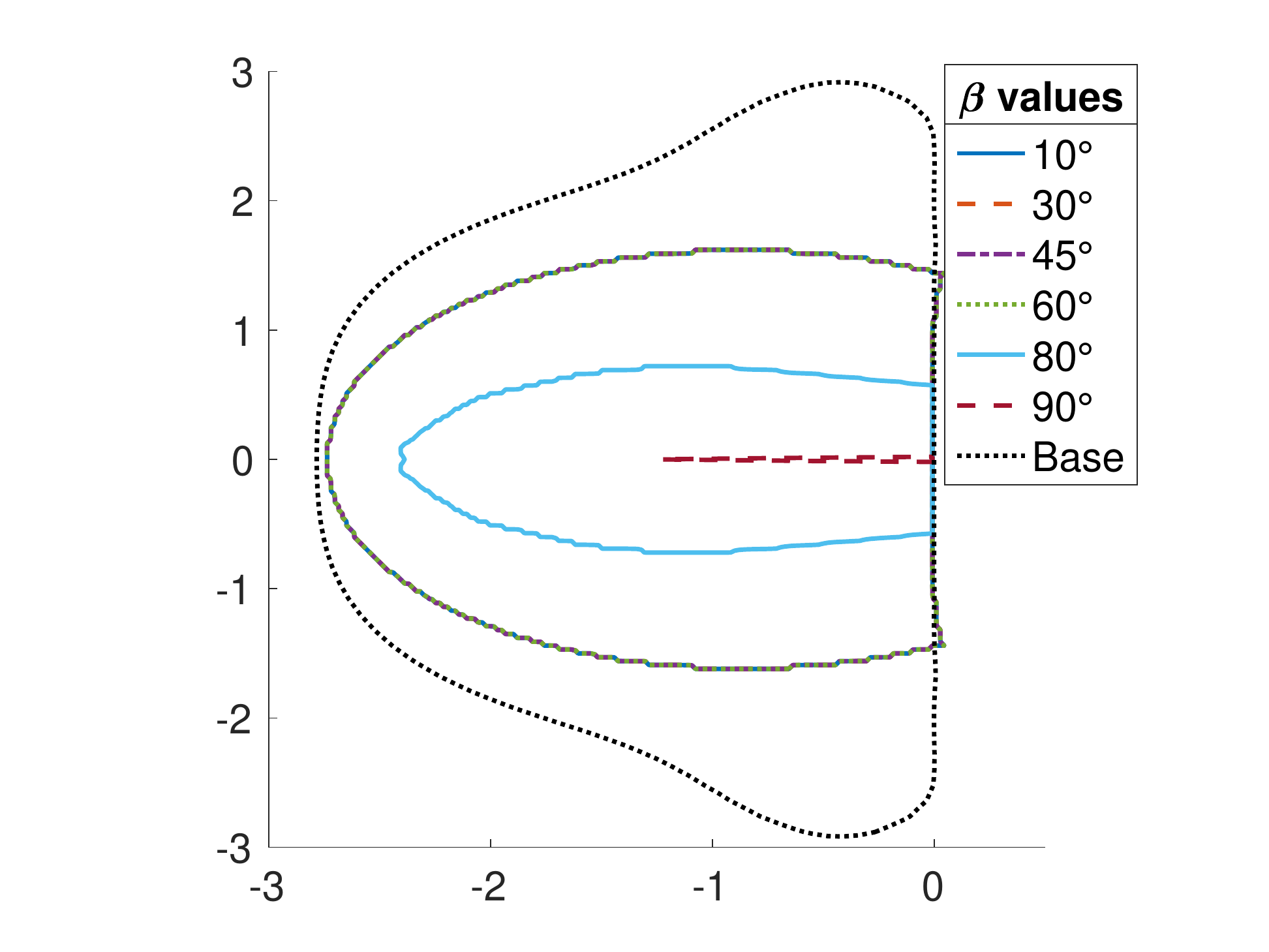}}
    \subfigure[$\mathcal{J}_{45^o,\beta}$ for IMEX-MRI-GARK3a]{\includegraphics[width = .49\textwidth,trim={2.5cm 0 1.5cm 0},clip]{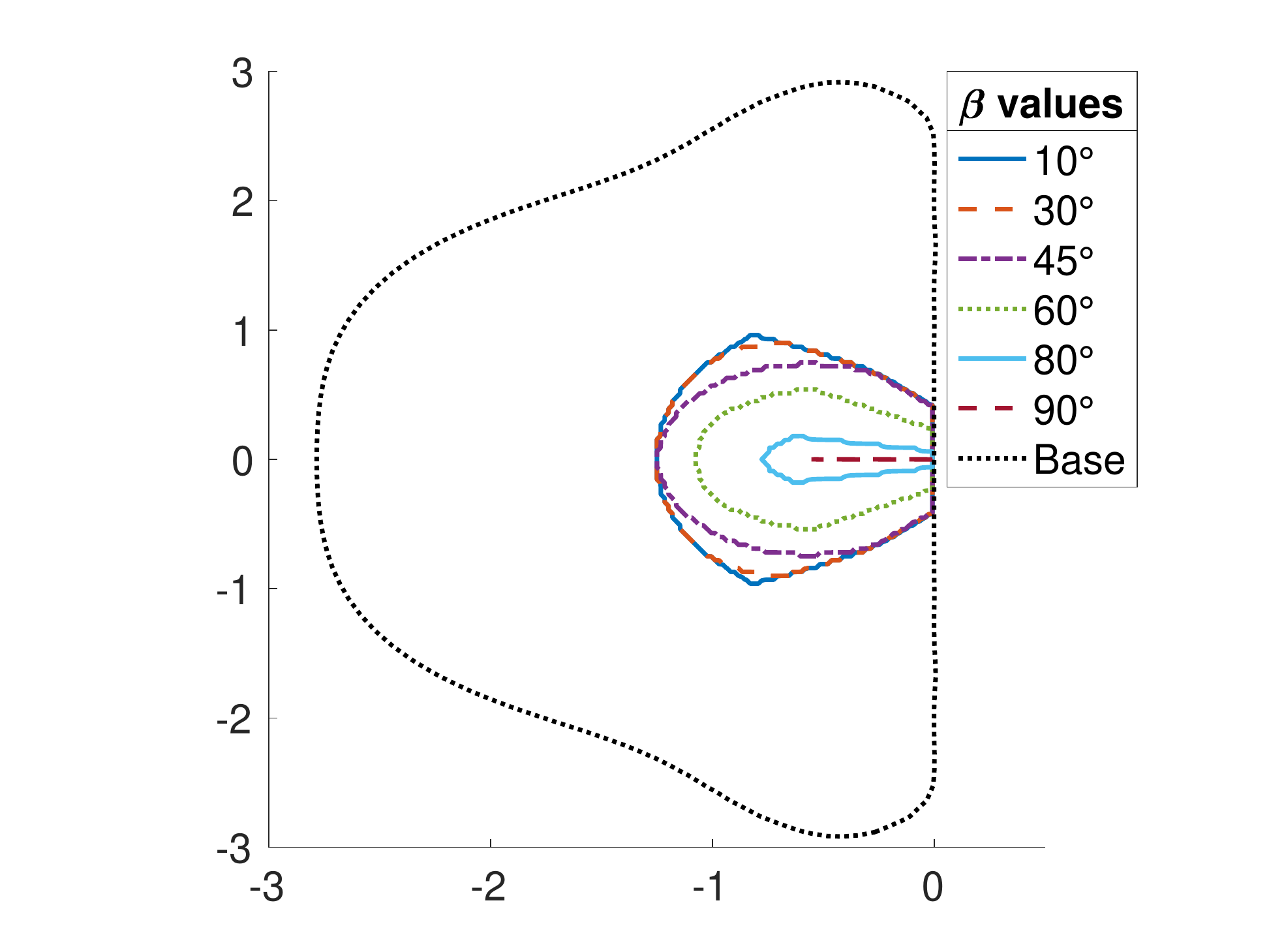}}
    \subfigure[$\mathcal{J}_{45^o,\beta}$ for IMEX-MRI-GARK3b]{\includegraphics[width = .49\textwidth,trim={2.5cm 0 1.5cm 0},clip]{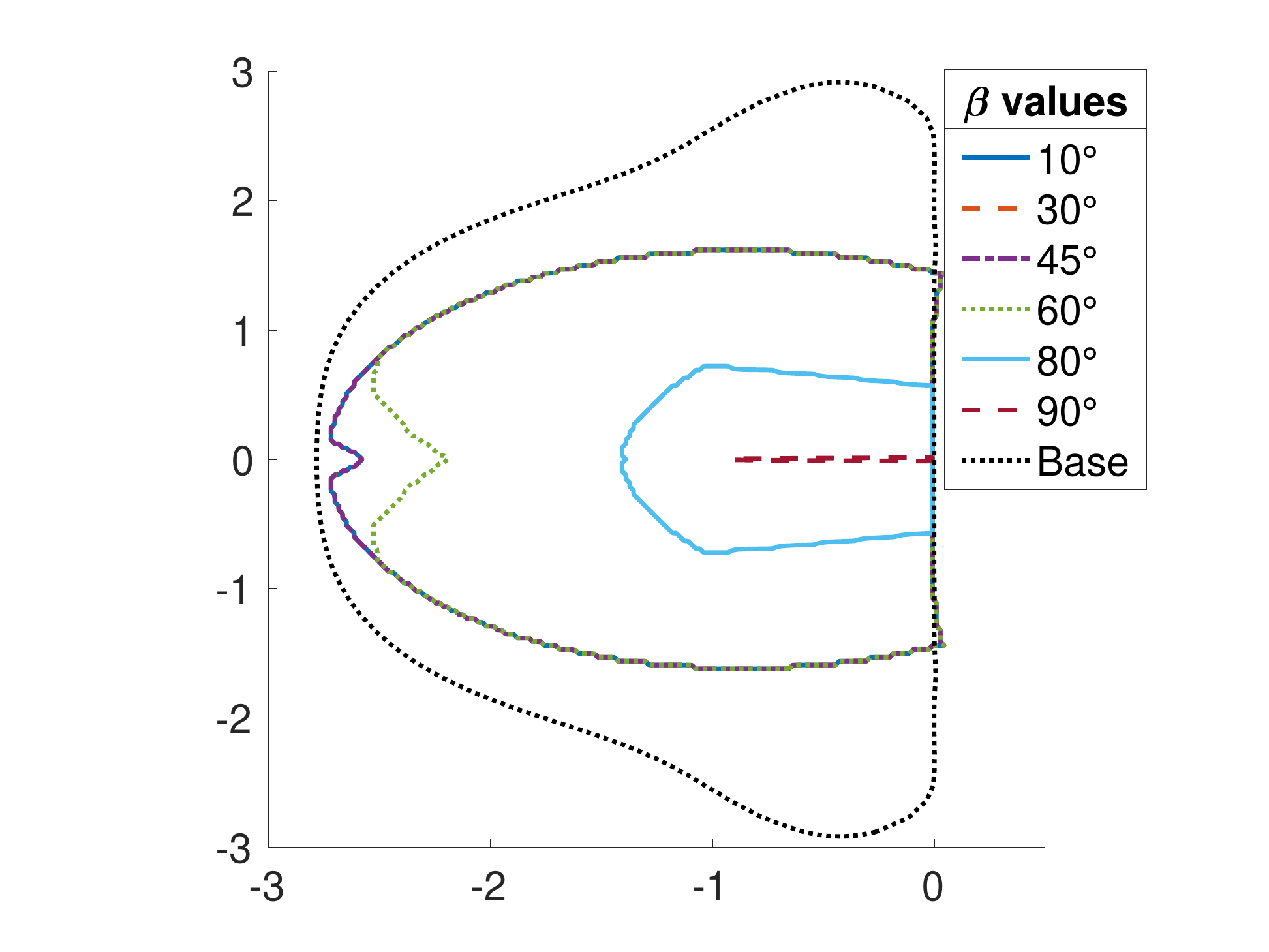}}
    \caption{Joint stability regions $\mathcal{J}_{\alpha,\beta}$ for both IMEX-MRI-GARK3a (left) and IMEX-MRI-GARK3b (right), at fast sector angles $\alpha=10^o$ (top) and $\alpha=45^o$ (bottom), for a variety of implicit sector angles $\beta$.  Each plot includes the joint stability region for the base IMEX-ARK table (shown as ``Base'').  The benefits of simultaneously optimizing the IMEX-MRI-GARK coefficients $\Gamma^{\{0\}}$ and $\Omega^{\{0\}}$ are clear, as $\mathcal{J}_{\alpha,\beta}$ for IMEX-MRI-GARK3b are significantly larger than those for IMEX-MRI-GARK3a.}
    \label{fig:imex-mri3-stability}
\end{figure}

In Figure \ref{fig:imex-mri3-stability} we plot the joint stability regions $\mathcal{J}_{\alpha,\beta}$ for both the IMEX-MRI-GARK3a and IMEX-MRI-GARK3b methods, for the fast time scale sectors $\mathcal{S}_{\alpha}^{\fast}$, $\alpha\in\{10^o,45^o\}$ and for the slow implicit sectors $\mathcal{S}_{\beta}^{\islow}$, $\beta\in\{10^o,30^o,45^o,60^o,80^o,90^o\}$.  In these figures we also plot the joint stability region for the slow base IMEX-ARK method, taken using the implicit slow wedge $\mathcal{S}_{90^o}^{\islow}$ (black dotted line). These results indicate that the joint stability regions for IMEX-MRI-GARK3a at each fast and implicit sector angle is significantly smaller than the base IMEX-ARK stability region.  Furthermore, these stability regions shrink considerably as the implicit sector angle $\beta$ grows from $10^o$ to $80^o$.  In contrast, the joint stability regions for IMEX-MRI-GARK3b are much larger, encompassing the majority of the base IMEX-ARK stability region for both fast sector angles $\alpha=10^o$ and $45^o$, and for implicit sector angles $\beta\le 60^o$, including a significant extent along the imaginary axis.  We therefore anticipate that this method should provide increased stability for IMEX multirate problems wherein advection comprises the slow explicit portion, as the corresponding Jacobian eigenvalues typically reside on the imaginary axis.

\subsection{Fourth-order Method}

We also constructed a fourth-order IMEX-MRI-GARK method using a base IMEX-ARK method of our own design (since we knew of no existing fourth-order method that satisfied our `sorted abscissae' requirement, $0 = c_1^{\slow} \leq \cdots \leq c_{\sslow}^{\slow} \leq 1$). To obtain IMEX-MRI-GARK4 we first converted our IMEX-ARK table to solve-decoupled form and then obtained the missing coefficients by satisfying internal consistency of the IMEX-ARK method. We then found the unknowns in $\Gamma^{\{0\}}$, $\Gamma^{\{1\}}$, $\Omega^{\{0\}}$ and $\Omega^{\{1\}}$ by solving the linear system resulting from \eqref{eq:AslowGARK}, \eqref{eq:internalcons}, \eqref{eq:order3} and \eqref{eq:order4} in MATLAB. The nonzero coefficients $\cbf^{\slow}$, $\Gamma^{\{0\}}$, $\Gamma^{\{1\}}$, $\Omega^{\{0\}}$ and $\Omega^{\{1\}}$ for this method, again accurate to 36 decimal digits, are given in Appendix \ref{sec:IMEX-MRI4}.

While this method indeed satisfies the full set of ARK consistency conditions \eqref{eq:AslowGARK}, internal consistency conditions \eqref{eq:internalcons}, third order conditions \eqref{eq:order3}, and fourth order conditions \eqref{eq:order4}, we have not yet been successful at optimizing its joint stability region $\mathcal{J}_{\alpha,\beta}$.  In fact, even when ignoring the slow explicit portion by setting $z^{\eslow}=0$ in our stability function \eqref{eq:stability_function}, the implicit+fast joint stability region is very small, rendering the full joint stability regions $\mathcal{J}_{\alpha,\beta}$ empty.  While we have already noted that this definition of joint stability is overly restrictive, and thus there may indeed be applications in which IMEX-MRI-GARK4 is suitable, we do not promote its widespread use, but include it here to demonstrate the predicted fourth-order convergence in our multirate example problems.


\section{Numerical Results}
\label{sec:numerical_results}

In this section we demonstrate the expected rates of convergence for the IMEX-MRI-GARK methods from Section \ref{sec:example_tables}.  Additionally, we compare the efficiency of the proposed methods against the legacy Lie--Trotter and Strang--Marchuk splittings \eqref{eq:first_order_split} and \eqref{eq:Strang_split}, as well as against two implicit MRI-GARK schemes from \cite{Sandu2019} of orders three and four, respectively: MRI-GARK-ESDIRK34a and MRI-GARK-ESDIRK46a. We consider two test problems: in Section \ref{subsec:kpr} we use a small Kv{\ae}rno-Prothero-Robinson (KPR) test problem to demonstrate the convergence of our methods, and in Section \ref{subsec:brusselator} we use a more challenging stiff `brusselator' test problem to investigate computational efficiency. Computations for the KPR problem were carried out in MATLAB while computations for the brusselator test were carried out in C using infrastructure from ARKODE, an ODE integration package within the SUNDIALS suite which provides explicit, implicit, and IMEX Runge--Kutta methods as well as MRI-GARK methods \cite{gardner2020enabling}. MATLAB implementations of both test problems are available in the public GitHub repository \cite{IMEXMRI_repo}.

\subsection{Kv{\ae}rno-Prothero-Robinson (KPR) Test}
\label{subsec:kpr}

We first consider the KPR test problem adapted from Sandu \cite{Sandu2019},
\begin{equation*}
  \begin{bmatrix} u \\ v \end{bmatrix}' = \mathbf{\Lambda} \begin{bmatrix} \frac{-3 + u^2 - \cos(\beta t)}{2u} \\ \frac{-2 + v^2 - \cos(t)}{2v} \end{bmatrix} - \begin{bmatrix} \frac{\beta \sin(\beta t)}{2u} \\ \frac{\sin(t)}{2v}\end{bmatrix}, \quad t\in\left[0,\tfrac{5\pi}{2}\right],
\end{equation*}
where
\[
  \mathbf{\Lambda} = \begin{bmatrix}\lambda^{\fast} & \frac{1-\varepsilon}{\alpha} (\lambda^{\fast} - \lambda^{\slow}) \\ -\alpha \varepsilon (\lambda^{\fast} - \lambda^{\slow}) & \lambda^{\slow} \end{bmatrix},
\]
and with initial conditions $u(0)=2$, $v(0)=\sqrt3$, corresponding to the exact solutions $u(t) = \sqrt{3 + \cos(\beta t)}$ and $v(t) = \sqrt{2 + \cos(t)}$.  Here, $u$ and $v$ correspond to the ``fast'' and ``slow'' solution variables, respectively.  We use the parameters $\lambda^{\fast} = -10$, $\lambda^{\slow} = -1$, $\varepsilon = 0.1$, $\alpha = 1$, $\beta = 20$. While this problem does not inherently require IMEX methods at the slow time scale, it is both nonlinear and non-autonomous, and has an analytical solution.  Thus it serves as an excellent problem to assess the convergence rates for the proposed IMEX-MRI-GARK methods.

We split this problem into the form \eqref{eq:probode} by setting each portion of the right hand side to be
\begin{align*}
    \feslow &= \begin{bmatrix}
       0 \\ \frac{\sin(t)}{2v}
    \end{bmatrix},\qquad
    \fislow = \begin{bmatrix} 0&0\\0&1\end{bmatrix} \mathbf{\Lambda} \begin{bmatrix} \frac{-3 + u^2 - \cos(\beta t)}{2u} \\ \frac{-2 + v^2 - \cos(t)}{2v} \end{bmatrix}, \qquad\text{and}\\
    \ffast &= \begin{bmatrix} 1&0\\0&0\end{bmatrix} \mathbf{\Lambda} \begin{bmatrix} \frac{-3 + u^2 - \cos{\beta t}}{2u} \\ \frac{-2 + v^2 - \cos(t)}{2v} \end{bmatrix} - \begin{bmatrix} \frac{\beta \sin(\beta t)}{2u} \\ 0 \end{bmatrix}.
\end{align*}
The slow component for implicit MRI-GARK methods is the sum of $\feslow$ and $\fislow$ which is then treated implicitly.

For the fast time scale of each method we use a step of size $h=\frac{H}{20}$, where we match the order of the inner solver with the overall method order: IMEX-MRI-GARK3 (a, b) and MRI-GARK-ESDIRK34a use the third-order explicit ``RK32'' from equation (233f) of \cite{butcherNumericalMethodsOrdinary2008},  IMEX-MRI-GARK4 and MRI-GARK-ESDIRK46a use the popular fourth-order explicit ``RK4'' method from \cite{kutta1901}, Strang--Marchuk uses the second-order explicit Heun method, and Lie--Trotter uses the explicit forward Euler method. For the implicit slow components of each method we use a standard Newton-Raphson nonlinear solver with dense Jacobian matrix and linear solver.

\begin{figure}[htbp]
    \centering
    {\includegraphics[scale=0.5]{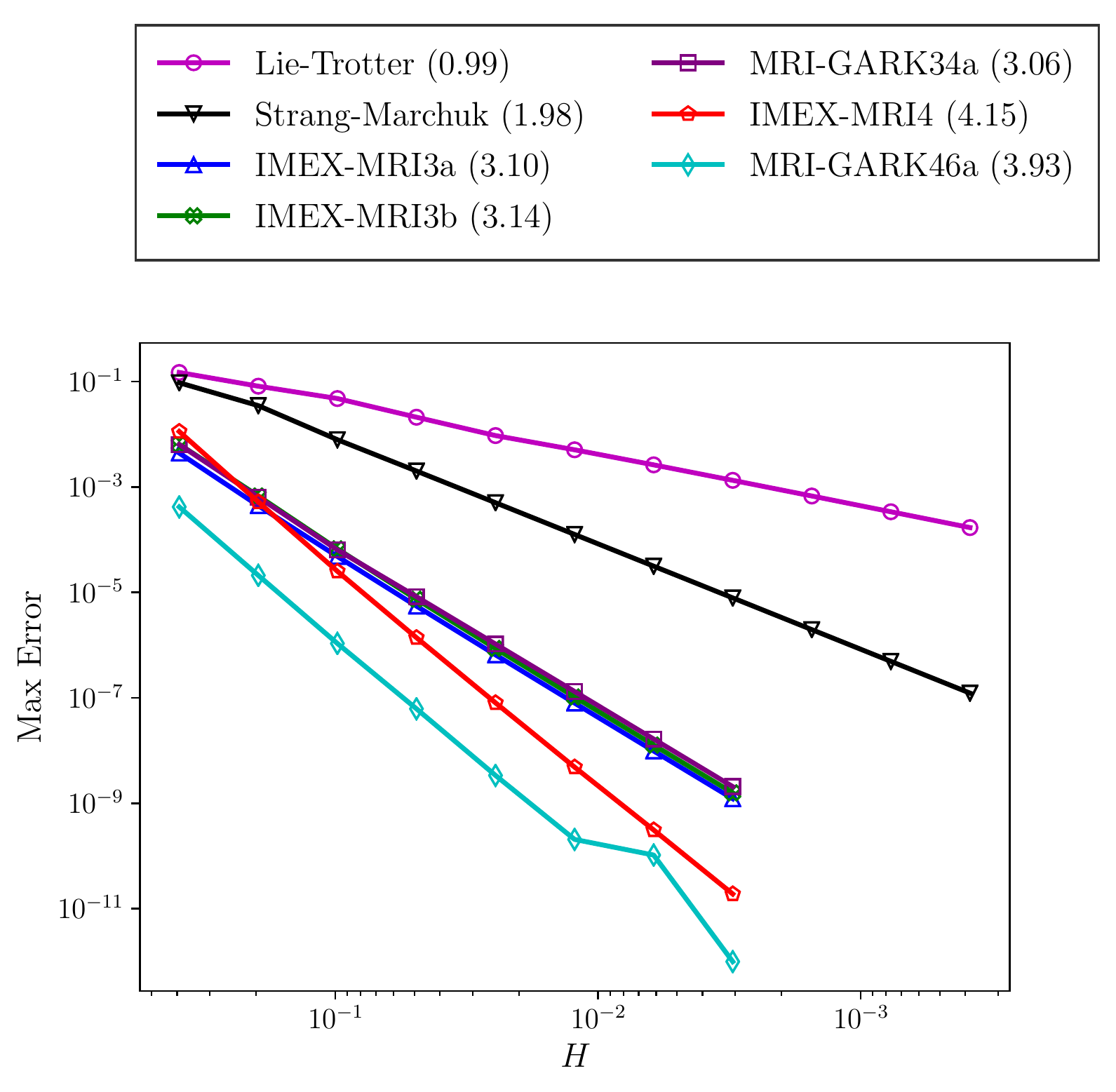}}
    \caption{Convergence for the KPR test problem from Section \ref{subsec:kpr}.  The measured convergence rates (given in parentheses) for each method match their theoretical predictions.}
    \label{fig:kpr_conv}
\end{figure}
In Figure \ref{fig:kpr_conv} we plot the maximum solution error over a set of 20 evenly-spaced temporal outputs in $[0,5\pi/2]$ for each method, at each of the slow step sizes $H = \pi/2^k$, for $k = 3,\ldots,10$ with IMEX-MRI-GARK and MRI-GARK methods and $k = 3,\ldots,13$ for the legacy methods. In the legend parentheses we show the overall estimated convergence rate, computed using a least-squares best fit of the $\log(\text{Max Error})$ versus $\log(H)$ results for each method.  For each method the theoretical order of convergence is reproduced.

\subsection{Brusselator Test}
\label{subsec:brusselator}
Our second, and more strenuous, test problem focuses on an advection-diffusion-reaction system of partial differential equations, as these are pervasive in computational physics and are typically solved using one of the two legacy methods \eqref{eq:first_order_split} or \eqref{eq:Strang_split}.  Here, both advection and diffusion may be evolved at the slow time scale, but due to their differential structure advection is typically treated explicitly, while diffusion is implicit.  Chemical reactions, however, frequently evolve on much faster time scales than advection and diffusion, and due to their nonlinearity and bound constraints (typically these are mass densities that must be non-negative), often require subcycling for both accuracy and stability.

We therefore consider the following example which is a stiff variation of the standard ``brusselator'' test problem \cite{hairerSolvingOrdinaryDifferential1993,hairerSolvingOrdinaryDifferential1996}:
\begin{align*}
    u_t &=  \alpha_u u_{xx} + \rho_u u_x + a - (w + 1)u + u^2v ,\\
    v_t &=  \alpha_v v_{xx} + \rho_v v_x + wu - u^2v,\\
    w_t &=  \alpha_w w_{xx} + \rho_w w_x + \frac{b-w}{\varepsilon} - wu,
\end{align*}
solved on $t \in [0,3]$ and $x \in [0,1]$, using stationary boundary conditions,
\[
    u_t(t,0) = u_t(t,1) = v_t(t,0) = v_t(t,1) = w_t(t,0) = w_t(t,1) = 0,
\]
and initial values,
\begin{align*}
    u(0,x) &= a + 0.1\sin(\pi x),\\
    v(0,x) &= b/a + 0.1\sin(\pi x),\\
    w(0,x) &= b + 0.1\sin(\pi x),
\end{align*}
with parameters $\alpha_{j} = 10^{-2}$, $\rho_{j} = 10^{-3}$, $a = 0.6$, $b = 2$, and $\varepsilon = 10^{-2}$.
We discretize these in space using a second order accurate centered difference approximation with $201$ or $801$ grid points.
As we do not have an analytical solution to this problem, we compute error by comparing against a reference solution generated using the same spatial grid, but that uses ARKODE's default fifth order diagonally implicit method with a time step of $H=10^{-7}$.

We split this problem into the form \eqref{eq:probode} by setting each portion of the right hand side to be the spatially-discretized versions of the operators
\begin{align*}
    \feslow = \begin{bmatrix}
         \rho_u u_x\\
         \rho_v v_x\\
         \rho_w w_x
    \end{bmatrix},\quad
    \fislow = \begin{bmatrix}
         \alpha_u u_{xx}\\
         \alpha_v v_{xx}\\
         \alpha_w w_{xx}
    \end{bmatrix},\quad\text{and}\quad
    \ffast = \begin{bmatrix}
        a - (w + 1)u + u^2v  \\
        wu - u^2v\\
       \frac{b-w}{\varepsilon} - wu
    \end{bmatrix}.
\end{align*}
The slow component for implicit MRI-GARK methods is the sum of $\feslow$ and $\fislow$ which is then treated implicitly.

We note that although this test problem indeed exhibits the same differential structure as large-scale advection-diffusion-reaction PDE models, a significant majority of those models are based on the compressible Navier--Stokes equations, wherein the `slow explicit' operator $\feslow$ would be nonlinear, would dominate the transport of reactants throughout the domain, and would be treated using a shock-capturing or essentially non-oscillatory spatial discretization.  Thus our results which follow should serve as only a simplified test problem for such scenarios, since in reality one would instead expect $\feslow$ to require a significantly larger share of the overall computational effort.  As a result, our subsequent results show only a `best case' scenario for implicit MRI-GARK methods, as implicit treatment of $\feslow$ in such large-scale applications is typically avoided due to its extreme cost and potential for nonlinear solver convergence issues.

For the subcycling portions of each method, we use a fast time step of $h=H/5$.  With the exception of Lie--Trotter we use fast implicit methods having accuracy equal to their corresponding multirate method: IMEX-MRI-GARK3 (a, b) and MRI-GARK-ESDIRK34a use the diagonally implicit method from Section 3.2.3 of \cite{condeImplicitImplicitExplicit2017} with $\beta = (3+\sqrt3)/6$, IMEX-MRI-GARK4 and MRI-GARK-ESDIRK46a use the diagonally implicit (5,3,4) method from \cite{cashDiagonallyImplicitRungeKutta1979}, while Lie--Trotter and Strang--Marchuk use an implicit second order method given by the Butcher table
$\begin{array}{c|cc} 1 & 1 & 0\\ 0 & -1 & 1\\ \hline & 1/2 & 1/2\end{array}$.
For both the implicit slow stages and the implicit fast stages we use a standard Newton-Raphson nonlinear solver with a banded direct linear solver.

\begin{figure}[htbp]
    \centering
    \includegraphics[width=\textwidth]{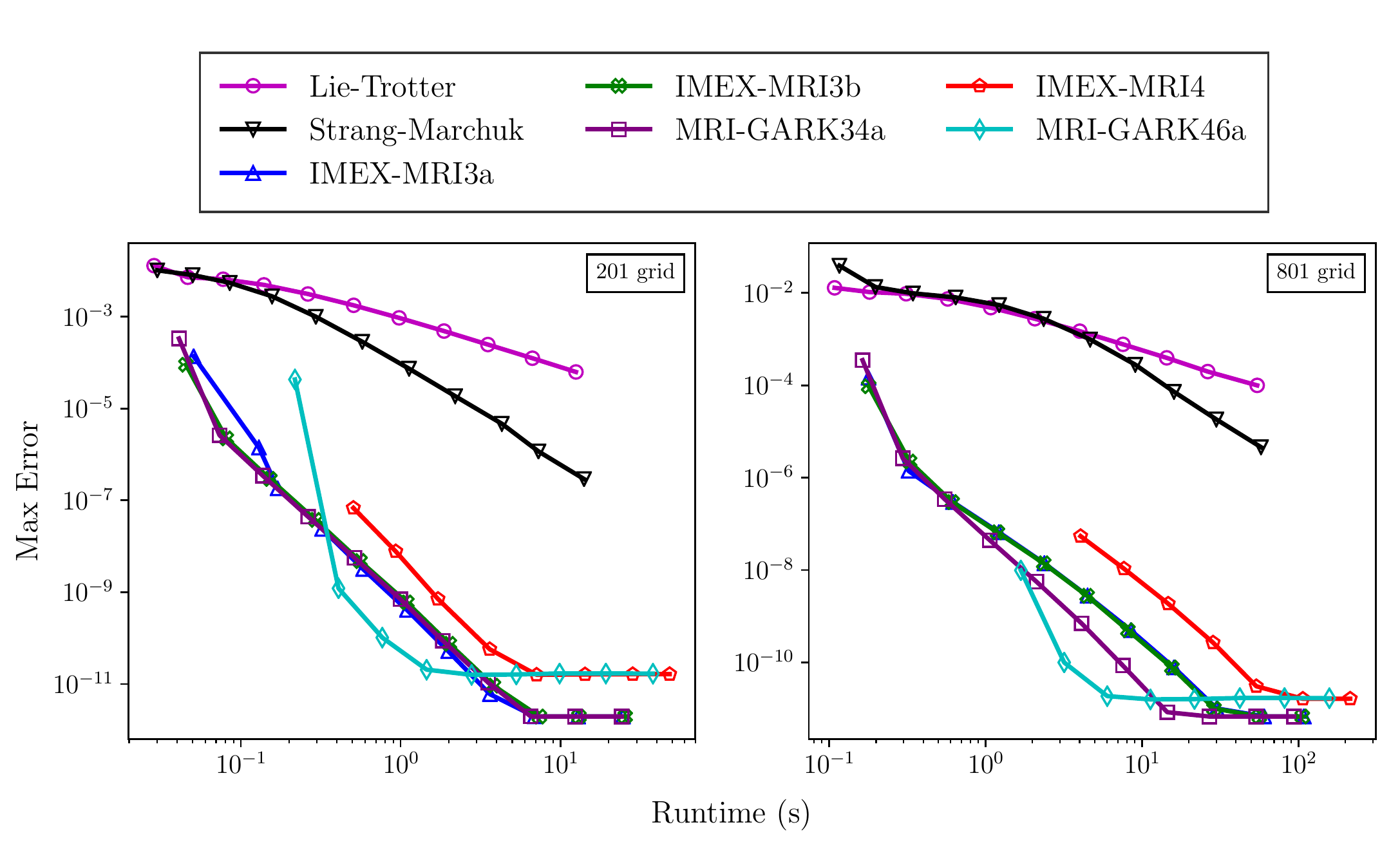}
    \caption{Efficiency for the stiff brusselator test problem from Section \ref{subsec:brusselator}, using 201 grid points (left) and 801 grid points (right). Best fit convergence rates on the 201 grid are $0.91, 1.92, 2.86, 2.92, 2.94, 3.12, 2.94$ for (Lie--Trotter, Strang--Marchuk, IMEX-MRI3a, IMEX-MRI3b, MRI-GARK34a, IMEX-MRI4, and MRI-GARK46a, resp.) and $0.90, 1.87, 2.41, 2.47, 3.02, 2.69, 2.42$ for the 801 grid.
    MRI-GARK46a and IMEX-MRI4 have limited stability on this test problem, with their curves missing for $H>1/40$ and $H>1/80$ respectively on the 201 grid, and $H>1/80$ and $H>1/160$ on the 801 grid. }
    \label{fig:brus_conv}
\end{figure}

For each spatial grid size in Figure \ref{fig:brus_conv}, we plot the runtimes and maximum solution error over a set of 10 evenly-spaced temporal outputs in $[0,3]$ for each method, at each of the slow step sizes $H=0.1 \cdot 2^{-k}$ for $k = 0,\ldots,10$.
We compute least squares fit convergence rates only on points within the asymptotic convergence regime, discarding points at larger $H$ values with higher than expected errors and points at smaller $H$ values where errors have already reached our reference solution accuracy.
We first note that as expected when applying Runge--Kutta methods to stiff applications, the measured convergence rates are slightly deteriorated from their theoretical peaks.  In addition to the challenges presented by stiffness, the reduced convergence for IMEX-MRI-GARK4 and MRI-GARK-ESDIRK46a is likely due to the limited reference solution accuracy of around $10^{-11}$. Additionally, the higher order methods experience order reduction when we increase the spatial grid size from 201 points to 801 points.

Furthermore, we point out that this problem highlights the reduced joint stability region for both the IMEX-MRI-GARK4 and MRI-GARK-ESDIRK46a methods, as the IMEX method was unstable for time step sizes larger than $H=1/80$ for 201 spatial grid points and larger than $H=1/160$ for the 801 spatial grid, while the implicit method was unstable for step sizes larger than $H=1/40$ and $H=1/80$ for 201 and 801 spatial grids respectively. All of the other methods were stable (if inaccurate) at even the largest step sizes tested.

Focusing our discussion on efficiency, at all accuracy levels shown in Figure \ref{fig:brus_conv}, IMEX-MRI-GARK and implicit MRI-GARK schemes are more efficient for this application than legacy approaches.  This is hardly surprising, due to their increased convergence rates and tighter coupling between the operators at the fast and slow time scales.
Comparing the third and fourth order IMEX-MRI-GARK methods, the third order methods are clearly more efficient for this test, which we believe results from three primary factors.  First, the third order methods require fewer slow implicit solves per step (3 vs 5).  Second, the fast-scale implicit Runge--Kutta methods used for both schemes have significantly different costs, with the third and fourth order methods requiring 2 and 5 implicit stages per step, respectively.  Both of these cost differences should be expected due to their differing method order; however the IMEX-MRI-GARK4 also experienced more severe order reduction for this problem, precluding those increased costs from being balanced by a significantly higher achievable convergence rate.

Expanding our consideration to include the full range of higher-order MRI-GARK approaches, MRI-GARK-ESDIRK46a is the most efficient at achieving tight desired accuracies (below $10^{-8}$), while all of the third-order methods were comparably efficient for larger accuracy levels.  For the 201 grid, there is no discernible difference in runtime between our IMEX-MRI-GARK3 a/b methods and MRI-GARK-ESDIRK34a; however MRI-GARK-ESDIRK34a achieves better efficiency for the 801 grid.  We recall, however, that due to the simple linear advection model in this problem, the results shown here represent a \emph{best case} scenario for implicit MRI-GARK methods, whereas the IMEX-MRI-GARK results should more accurately reflect their expected performance on large-scale reactive flow problems.  We thus anticipate that when applied to the targeted large-scale applications, the IMEX-MRI-GARK3 a/b methods will prove to be significantly more efficient, due to their combination of excellent convergence and flexibility in allowing explicit treatment of $\feslow$.

We finally note that of the methods that allow the originally-desired IMEX + multirate treatment of this problem (i.e., \emph{not} including the implicit MRI-GARK methods), the proposed IMEX-MRI-GARK methods enable accuracies that would otherwise be intractable with lower-order approaches.

\section{Conclusions}

In this paper we have introduced a new class of multirate integration methods that support implicit-explicit treatment of the slow time scale.  These IMEX-MRI-GARK methods are highly-flexible: in addition to supporting IMEX treatment of the slow time scale, the fast time scale is only assumed to be solved using another sufficiently-accurate approximation, thereby allowing for the fast time scale to be further decomposed into a mix of implicit and explicit components, or even into a multirate method itself.  As with their related non-IMEX MRI-GARK counterparts \cite{Sandu2019}, the coupling from slow to fast time scale occurs through modification of the fast time-scale function $f^{\fast}(t,y)$ to include a polynomial forcing term, $g(t)$, that incorporates slow time scale tendencies into the fast time scale dynamics.

In addition to defining IMEX-MRI-GARK methods, we have provided rigorous derivation of conditions on their coefficients to guarantee orders three and four.  Furthermore, we have provided the corresponding linear stability function for IMEX-MRI-GARK methods, and extended the definition of ``joint stability'' from Zharovsky et al. \cite{Zharovsky2015Class} to accommodate a three-component additive splitting.

With these theoretical foundations, we have presented three specific IMEX-MRI-GARK methods, two third order methods derived from Ascher, Ruuth and Spiteri's `(3,4,3)' ARK method \cite{ascherImplicitexplicitRungeKuttaMethods1997}, and one fourth order method of our own design.


We then provided asymptotic convergence results for the three proposed methods, using the standard Kv{\ae}rno-Prothero-Robinson (KPR) multirate test problem, where each method exhibited its expected convergence rate.  To assess method efficiency, we utilized a more challenging stiff brusselator PDE test problem, which showed that the proposed methods were uniformly more efficient than the lower-order Lie--Trotter and Strang--Marchuk methods at all accuracy levels tested.  Moreover, although such methods cannot allow for IMEX treatment of the slow time scale (and thus efficiency comparisons are somewhat artificial), we also compared against recently-proposed implicit MRI-GARK methods \cite{Sandu2019}.  Here, our third order IMEX-MRI-GARK methods proved competitive, but the higher cost per step of our fourth order IMEX-MRI-GARK method rendered it the least efficient of the group.

We note that much work remains.  For starters, we plan to derive new fourth-order IMEX-MRI-GARK methods with optimized linear stability regions and with a decreased cost per step.  We anticipate that this will require simultaneous derivation of both the base IMEX-ARK method \emph{and} its IMEX-MRI-GARK extension, due to the tight interplay between these methods and their joint stability.  An obvious (yet tedious) extension of this work would be to derive the order conditions for fifth-order IMEX-MRI-GARK methods, and to construct tables to implement such approaches.  Additionally, we would like to create new IMEX-MRI-GARK methods that include embeddings, thereby allowing for robust temporal adaptivity at both the slow and fast time scales.  While extension of the IMEX-MRI-GARK algorithm to include an alternate set of IMEX-ARK embedding coefficients is straightforward, creation of optimal embedded multirate methods and fast/slow temporal adaptivity controllers have barely been touched in the literature.  Finally, we anticipate the creation of `solve-coupled' IMEX-MRI-GARK and MRI-GARK methods, and the accompanying work on efficient nonlinear solvers, to allow a tighter coupling between implicit and fast processes in these multirate approaches.


\section*{Acknowledgments}
The authors would like to thank David Gardner, Carol Woodward and John Loffeld for their insightful discussions throughout the derivation of this work.  We would also like to thank the SMU Center for Research Computing for use of the \emph{Maneframe2} computing cluster, where we performed all simulations reported in this work.

\appendix

\section{IMEX-MRI-GARK3a}
\label{sec:IMEX-MRI3a}

The nonzero coefficients for IMEX-MRI-GARK3a (accurate to 36 decimal digits) are:
\small{\begin{align*}
         c^{\slow}_1 &= 0,\\
         c^{\slow}_2 &= c^{\slow}_3 = 0.4358665215084589994160194511935568425,\\
         c^{\slow}_4 &= c^{\slow}_5 = 0.7179332607542294997080097255967784213,\\
         c^{\slow}_6 &= c^{\slow}_7 = c^{\slow}_8 = 1,\\
         \\
         \gamma^{\{0\}}_{2,1} &= -\gamma^{\{0\}}_{3,1} = \gamma^{\{0\}}_{3,3} = \gamma^{\{0\}}_{5,5} = \gamma^{\{0\}}_{6,1} = -\gamma^{\{0\}}_{7,1} = \gamma^{\{0\}}_{7,7} \\
         &= 0.4358665215084589994160194511935568425,\\
         \gamma^{\{0\}}_{4,1} &= -\gamma^{\{0\}}_{5,1} = -0.4103336962288525014599513720161078937,\\
         \gamma^{\{0\}}_{4,3} &= 0.6924004354746230017519416464193294724,\\
         \gamma^{\{0\}}_{5,3} &= -0.8462002177373115008759708232096647362,\\
         \gamma^{\{0\}}_{6,3} &= 0.9264299099302395700444874096601015328,\\
         \gamma^{\{0\}}_{6,5} &= -1.080229692192928069168516586450436797,\\
         \\
         \omega^{\{0\}}_{2,1} &= \omega^{\{0\}}_{8,7} = 0.4358665215084589994160194511935568425,\\
         \omega^{\{0\}}_{4,1} &= -0.5688715801234400928465032925317932021,\\
         \omega^{\{0\}}_{4,3} &= 0.8509383193692105931384935669350147809,\\
         \omega^{\{0\}}_{5,1} &= -\omega^{\{0\}}_{5,3} = 0.454283944643608855878770886900124654,\\
         \omega^{\{0\}}_{6,1} &= -0.4271371821005074011706645050390732474,\\
         \omega^{\{0\}}_{6,3} &= 0.1562747733103380821014660497037023496,\\
         \omega^{\{0\}}_{6,5} &= 0.5529291480359398193611887297385924765,\\
         \omega^{\{0\}}_{8,1} &= 0.105858296071879638722377459477184953,\\
         \omega^{\{0\}}_{8,3} &= 0.655567501140070250975288954324730635,\\
         \omega^{\{0\}}_{8,5} &= -1.197292318720408889113685864995472431.
\end{align*}}
We note that these coefficients (and all of those that follow) are available electronically in \cite{IMEXMRI_repo}.

\section{IMEX-MRI-GARK3b}
\label{sec:IMEX-MRI3b}

The nonzero coefficients for IMEX-MRI-GARK3b (accurate to 36 decimal digits) are:
\small{\begin{align*}
         c^{\slow}_1 &= 0,\\
         c^{\slow}_2 &= c^{\slow}_3 = 0.4358665215084589994160194511935568425,\\
         c^{\slow}_4 &= c^{\slow}_5 = 0.7179332607542294997080097255967784213,\\
         c^{\slow}_6 &= c^{\slow}_7 = c^{\slow}_8 = 1,\\
         \\
         \gamma^{\{0\}}_{2,1} &= -\gamma^{\{0\}}_{3,1} = \gamma^{\{0\}}_{3,3} = \gamma^{\{0\}}_{5,5} = \gamma^{\{0\}}_{7,7} \\
                      &= 0.4358665215084589994160194511935568425,\\
         \gamma^{\{0\}}_{4,1} &= -\gamma^{\{0\}}_{5,1} = 0.0414273753564414837153799230278275639,\\
         \gamma^{\{0\}}_{4,3} &= 0.2406393638893290165766103513753940148\\
         \gamma^{\{0\}}_{5,3} &= -0.3944391461520175157006395281657292786\\
         \gamma^{\{0\}}_{6,1} &= -\gamma^{\{0\}}_{7,1} = 0.1123373143006047802633543416889605123\\
         \gamma^{\{0\}}_{6,3} &= 1.051807513648115027700693049638099167\\
         \gamma^{\{0\}}_{6,5} &= -0.8820780887029493076720571169238381009\\
         \gamma^{\{0\}}_{7,3} &= -0.1253776037178754576562056399779976346\\
         \gamma^{\{0\}}_{7,5} &= -0.1981516034899787614964594695265986957\\
         \\
         \omega^{\{0\}}_{2,1} &= \omega^{\{0\}}_{8,7} = 0.4358665215084589994160194511935568425,\\
         \omega^{\{0\}}_{4,1} &= -0.1750145285570467590610670000018749059,\\
         \omega^{\{0\}}_{4,3} &= 0.4570812678028172593530572744050964846,\\
         \omega^{\{0\}}_{5,1} &= -\omega^{\{0\}}_{5,3} = 0.06042689307721552209333459437020635774,\\
         \omega^{\{0\}}_{6,1} &= 0.1195213959425454440038786034027936869,\\
         \omega^{\{0\}}_{6,3} &= -1.84372522668966191789853395029629765,\\
         \omega^{\{0\}}_{6,5} &= 2.006270569992886974186645621296725542,\\
         \omega^{\{0\}}_{7,1} &= -0.5466585780430528451745431084418669343,\\
         \omega^{\{0\}}_{7,3} &= 2,\\
         \omega^{\{0\}}_{7,5} &= -1.453341421956947154825456891558133066,\\
         \omega^{\{0\}}_{8,1} &= 0.105858296071879638722377459477184953,\\
         \omega^{\{0\}}_{8,3} &= 0.655567501140070250975288954324730635,\\
         \omega^{\{0\}}_{8,5} &= -1.197292318720408889113685864995472431.
\end{align*}}

\section{IMEX-MRI-GARK4}
\label{sec:IMEX-MRI4}

The nonzero coefficients for IMEX-MRI-GARK4 (accurate to 36 decimal digits) are:
\small{\begin{align*}
         \cbf^{\slow} &= \left[\begin{array}{ccccccccccccc} 0 & \frac12 & \frac12 & \frac58 & \frac58 & \frac34 & \frac34 & \frac78 & \frac78 & 1 & 1 & 1 \end{array}\right],\\
         \gamma^{\{0\}}_{2,1} &= \tfrac{1}{2},  \\
         \gamma^{\{0\}}_{3,1} &= -\gamma^{\{0\}}_{3,3} = -\gamma^{\{0\}}_{5,5} = -\gamma^{\{0\}}_{7,7} = -\gamma^{\{0\}}_{9,9} = -\gamma^{\{0\}}_{11,11} = -\tfrac{1}{4},  \\
         \gamma^{\{0\}}_{4,1} &= -3.97728124810848818306703385146227889,\\
         \gamma^{\{0\}}_{4,3} &= 4.10228124810848818306703385146227889,\\
         \gamma^{\{0\}}_{5,1} &= -0.0690538874140169123272414708480937406,  \\
         \gamma^{\{0\}}_{5,3} &= -0.180946112585983087672758529151906259,  \\
         \gamma^{\{0\}}_{6,1} &= -1.76176766375792052886337896482241241,\\
         \gamma^{\{0\}}_{6,3} &= 2.69452469837729861015533815079146138,\\
         \gamma^{\{0\}}_{6,5} &= -0.807757034619378081291959185969048978,\\
         \gamma^{\{0\}}_{7,1} &= 0.555872179155396948730508100958808496, \\
         \gamma^{\{0\}}_{7,3} &= -0.679914050157999501395850152788348695, \\
         \gamma^{\{0\}}_{7,5} &= -\gamma^{\{0\}}_{8,5} = -0.125958128997397447334657948170459801, \\
         \gamma^{\{0\}}_{8,1} &= -5.84017602872495595444642665754106511,\\
         \gamma^{\{0\}}_{8,3} &= 8.17445668429191508919127080571071637,\\
         \gamma^{\{0\}}_{8,7} &= -2.33523878456435658207950209634011106, \\
         \gamma^{\{0\}}_{9,1} &= -1.9067926451678118080947593050360523, \\
         \gamma^{\{0\}}_{9,3} &= -\gamma^{\{0\}}_{10,3} = -1.54705781138512393363298457924938844 \\
         \gamma^{\{0\}}_{9,5} &= -\gamma^{\{0\}}_{10,5} = 4.12988801314935030595449173802031322, \\
         \gamma^{\{0\}}_{9,7} &= -\gamma^{\{0\}}_{10,7} = -0.926037556596414564226747853734872477, \\
         \gamma^{\{0\}}_{10,1} &= 3.33702815168872605455765278252966252,\\
         \gamma^{\{0\}}_{10,9} &= -1.55523550652091424646289347749361021, \\
         \gamma^{\{0\}}_{11,1} &= -0.821293629221007618720524112312446752, \\
         \gamma^{\{0\}}_{11,3} &= 0.328610356068599988551677264268969646, \\
         \gamma^{\{0\}}_{11,5} &= 0.678001812102026694142641232421139516, \\
         \gamma^{\{0\}}_{11,7} &= -0.342779287862800022896645471462060708, \\
         \gamma^{\{0\}}_{11,9} &= -0.0925392510868190410771489129156017025, \\
         \\
         \gamma^{\{1\}}_{4,1} &= -\gamma^{\{1\}}_{4,3} = 8.70456249621697636613406770292455778,\\
         \gamma^{\{1\}}_{6,1} &= 3.91164310234387488238124087134101229,\\
         \gamma^{\{1\}}_{6,3} &= -5.02715717158263104496515924327911025,\\
         \gamma^{\{1\}}_{6,5} &= 1.11551406923875616258391837193809796,\\
         \gamma^{\{1\}}_{8,1} &= 10.8186076991391180114318371131645132,\\
         \gamma^{\{1\}}_{8,3} &= -14.9890852682678311755908413058447354,\\
         \gamma^{\{1\}}_{8,7} &= 4.17047756912871316415900419268022213, \\
         \gamma^{\{1\}}_{10,1} &= -2.61047101304182849292578695498722043, \\
         \gamma^{\{1\}}_{10,9} &= 2.61047101304182849292578695498722043, \\
         \\
         \omega^{\{0\}}_{2,1} &= \tfrac{1}{2}, \\
         \omega^{\{0\}}_{4,1} &= -1.91716534363662868878172216064946905,\\
         \omega^{\{0\}}_{4,3} &= 2.04216534363662868878172216064946905,\\
         \omega^{\{0\}}_{5,1} &= -\omega^{\{0\}}_{5,3} = -0.404751031801105942697915907046990469, \\
         \omega^{\{0\}}_{6,1} &= 11.4514660224922163666569802860263173,\\
         \omega^{\{0\}}_{6,3} &= -30.2107574752650427144064781557395061,\\
         \omega^{\{0\}}_{6,5} &= 18.8842914527728263477494978697131888,\\
         \omega^{\{0\}}_{7,1} &= -0.709033564760261450684711672946330144,\\
         \omega^{\{0\}}_{7,3} &= 1.03030720858751876652616190884004718,\\
         \omega^{\{0\}}_{7,5} &= -\omega^{\{0\}}_{8,5} = -0.321273643827257315841450235893717036, \\
         \omega^{\{0\}}_{8,1} &= -29.9954871645582843984091068494419927,\\
         \omega^{\{0\}}_{8,3} &= 37.605982774991801805364896856243857,\\
         \omega^{\{0\}}_{8,7} &= -7.80676925426077472279724024269558129, \\
         \omega^{\{0\}}_{9,1} &= 3.10466505427296211633876939184912422,\\
         \omega^{\{0\}}_{9,3} &= -\omega^{\{0\}}_{10,3} = -2.43032501975716229713206592741556636,\\
         \omega^{\{0\}}_{9,5} &= -\omega^{\{0\}}_{10,5} = -1.90547930115152463521920165948384213,\\
         \omega^{\{0\}}_{9,7} &= -\omega^{\{0\}}_{10,7} = 1.23113926663572481601249819505028427, \\
         \omega^{\{0\}}_{10,1} &= -2.42442954775204786987587591435551401,\\
         \omega^{\{0\}}_{10,9} &= -0.555235506520914246462893477493610215,\\
         \omega^{\{0\}}_{11,1} &= -0.010441350444797485902945189451653542,\\
         \omega^{\{0\}}_{11,3} &= 0.0726030361465507450515210450548814161,\\
         \omega^{\{0\}}_{11,5} &= -0.128827595167726095223945409857642431,\\
         \omega^{\{0\}}_{11,7} &= 0.112935535009382356613944010712215408,\\
         \omega^{\{0\}}_{11,9} &= \omega^{\{0\}}_{12,9} = -0.0462696255434095205385744564578008512,\\
         \omega^{\{0\}}_{12,1} &= -0.81085227877621013281757892286079321,\\
         \omega^{\{0\}}_{12,3} &= 0.25600731992204924350015621921408823,\\
         \omega^{\{0\}}_{12,5} &= 0.806829407269752789366586642278781947,\\
         \omega^{\{0\}}_{12,7} &= -0.455714822872182379510589482174276116,\\
         \omega^{\{0\}}_{12,11} &= \tfrac{1}{4}  \\
         \\
         \omega^{\{1\}}_{4,1} &= -\omega^{\{1\}}_{4,3}= 4.0843306872732573775634443212989381,\\
         \omega^{\{1\}}_{6,1} &= -21.8434299813822208479181287579586536,\\
         \omega^{\{1\}}_{6,3} &= 59.6120128869278735434171244973850312,\\
         \omega^{\{1\}}_{6,5} &= -37.7685829055456526954989957394263776,\\
         \omega^{\{1\}}_{8,1} &= 61.6590414586370916981876370447766458,\\
         \omega^{\{1\}}_{8,3} &= -77.2725799671586411437821175301678084,\\
         \omega^{\{1\}}_{8,7} &= 15.6135385085215494455944804853911626,\\
         \omega^{\{1\}}_{10,1} &= -\omega^{\{1\}}_{10,9} = -1.11047101304182849292578695498722043.
\end{align*}}

\bibliographystyle{siamplain}
\bibliography{references}
\end{document}